\documentclass[a4paper]{amsart}
\usepackage{amsthm}
\usepackage{amssymb}
\usepackage{dsfont}
\usepackage{mathrsfs}
\usepackage{tikz-cd}
\usepackage{enumerate}
\usepackage{eucal}
\usepackage{textgreek}
\usepackage{hyperref}
\hypersetup{%
  bookmarksnumbered=true,%
  colorlinks=true,%
  linkcolor=blue,%
  citecolor=blue,%
  filecolor=blue,%
  menucolor=blue,%
  urlcolor=blue,%
  bookmarksopen=true,%
  bookmarksdepth=2,%
  pageanchor=true}

\usepackage{mathtools}
\usepackage{todonotes}

\numberwithin{equation}{section}

\theoremstyle{plain}
\newtheorem{theorem}[equation]{Theorem}
\newtheorem{proposition}[equation]{Proposition}
\newtheorem{lemma}[equation]{Lemma} 
\newtheorem{corollary}[equation]{Corollary}

\theoremstyle{definition}
\newtheorem{definition}[equation]{Definition}
\newtheorem{example}[equation]{Example}

\theoremstyle{remark}
\newtheorem{remark}[equation]{Remark} 
\newtheorem{question}[equation]{Question}
\newtheorem*{ack}{Acknowledgements}

\newcommand{\art}{\operatorname{art}}
\newcommand{\add}{\operatorname{add}}
\newcommand{\bcd}{D_\fp}
\newcommand{\bSpec}{\operatorname{Spc}}

\newcommand{\colim}{\operatorname*{colim}}

\newcommand{\cat}{\mathcal}
\newcommand{\comp}[1]{{#1}^{\mathrm{c}}}
\newcommand{\cone}{\operatorname{cone}}

\newcommand{\dbcat}[1]{{\mathbf D}^{\mathrm b}(\operatorname{mod}#1)}
\newcommand{\dcat}[1]{\mathbf{D}(#1)}
\newcommand{\End}{\operatorname{End}}
\newcommand{\Ext}{\operatorname{Ext}}

\newcommand{\gam}{\varGamma}

\newcommand{\Hom}{\operatorname{Hom}}
\newcommand{\fHom}{\operatorname{\mathcal{H}\!\!\;\mathit{om}}}

\newcommand{\ik}{\mathbf{i}k}
\newcommand{\bfi}{\mathbf{i}}

\newcommand{\Inj}{\operatorname{Inj}}
\newcommand{\kos}[2]{{#1}/\!\!/{#2}}

\newcommand{\Ker}{\operatorname{Ker}}
\newcommand{\KInj}[1]{\mathbf K(\Inj #1)}
\newcommand{\lam}{\varLambda}
\newcommand{\length}{\operatorname{length}}
\newcommand{\loc}{L}
\newcommand{\lotimes}{\otimes^{\mathbf L}}
\newcommand{\lra}{\longrightarrow}

\newcommand{\Mod}{\operatorname{Mod}}
\newcommand{\noeth}{\operatorname{noeth}}
\newcommand{\one}{\mathds 1}
\newcommand{\op}{\mathrm{op}}
\newcommand{\Proj}{\operatorname{Proj}}

\newcommand{\uHom}{\underline{\Hom}}
\newcommand{\Si}{\Sigma} 

\newcommand{\Soc}{\operatorname{Soc} }
\newcommand{\Spec}{\operatorname{Spec}}
\newcommand{\StMod}{\operatorname{StMod}}
\newcommand{\stmod}{\operatorname{stmod}}
\newcommand{\supp}{\operatorname{supp}}
\newcommand{\swd}{D^{\scriptscriptstyle{\mathrm {SW}}}}

\newcommand{\thick}{\operatorname{thick}}

\newcommand*\cocolon{%
        \nobreak
        \mskip6mu plus1mu
        \mathpunct{}%
        \nonscript
        \mkern-\thinmuskip
        {:}%
        \mskip2mu
        \relax
}

\newcommand{\iso}{\xrightarrow{\raisebox{-.6ex}[0ex][0ex]{$\scriptstyle{\sim}$}}}
\newcommand{\longiso}{\xrightarrow{\ \raisebox{-.6ex}[0ex][0ex]{$\scriptstyle{\sim}$}\ }}
\newcommand{\mhat}{{}^{\wedge}_\fm}

\newcommand{\phat}{{}^{\wedge}_\fp}
\newcommand{\xra}{\xrightarrow}

\newcommand{\bfD}{\mathbf D} 

\newcommand{\bbP}{\mathbb P}

\newcommand{\bsr}{\boldsymbol{r}}

\newcommand{\bbZ}{\mathbb Z}

\newcommand{\fa}{\mathfrak{a}} 
\newcommand{\fm}{\mathfrak{m}} 
\newcommand{\fp}{\mathfrak{p}}
\newcommand{\fq}{\mathfrak{q}} 
\newcommand{\eps}{\varepsilon}

\definecolor{myblue}{RGB}{100,100,255}

\title[Locally dualisable modular representations]{Locally dualisable modular representations \\ and local regularity}

\author[Benson, Iyengar, Krause, and Pevtsova]{Dave Benson, Srikanth
  B. Iyengar, Henning Krause \\ and Julia Pevtsova}

\address{Dave Benson \\ 
Institute of Mathematics\\ 
University of Aberdeen\\ 
King's College\\ 
Aberdeen AB24 3UE\\ 
Scotland U.K.}

\address{Srikanth B. Iyengar\\ 
Department of Mathematics\\
University of Utah\\ 
Salt Lake City, UT 84112\\ 
U.S.A.}

\address{Henning Krause\\ 
Fakult\"at f\"ur Mathematik\\ 
Universit\"at Bielefeld\\ 
33501 Bielefeld\\ 
Germany.}

\address{Julia Pevtsova\\ 
Department of Mathematics\\ 
University of Washington\\ 
Seattle, WA 98195\\ 
U.S.A.}

\begin{document}

\begin{abstract} 
This work concerns the stable module category of a finite group over
a field of characteristic dividing the group order.  The minimal
localising tensor ideals correspond to the non-maximal
homogeneous prime ideals in the cohomology ring of the group. Given such a
prime ideal, a number of characterisations of the dualisable objects in the corresponding
tensor ideal are given.  One characterisation of interest is that they are exactly the
modules whose restriction along a corresponding $\pi$-point
are finite dimensional plus projective.  A key insight is the identification of a special property of the stable module category that controls the cohomological behaviour of local dualisable objects. This property, introduced in this work for general triangulated categories and called  local regularity,  is  related to strong generation. A major part of the paper is devoted to developing this notion and investigating its ramifications for various special classes of objects in tensor triangulated categories.
\end{abstract}

\keywords{dualisable object, finite length object, local regularity, $\pi$-point, stable module category, tensor triangulated category}

\subjclass[2020]{20C20 (primary); 18G80, 20J06 (secondary)}

\date{\today}

\maketitle

\setcounter{tocdepth}{1}
\tableofcontents

\section{Introduction}
Let $G$ be a finite group and $k$ a field of characteristic $p$.
In modular representation theory, traditionally one
restricted attention to finitely generated $kG$-modules.
These have a number of desirable properties that make them
easier to work with, such as the Krull--Remak--Schmidt Theorem:
every finitely generated module is a direct sum of indecomposables;
the indecomposables have local endomorphism rings, and the
isomorphism types and multiplicities of the direct summands are
uniquely determined. This implies cancellation: if $M\oplus X
\cong N \oplus X$ then $M\cong N$. Krull--Remak--Schmidt drives the theory of
vertices and sources, as well as various other well known developments in
the subject.

In the nineties and subsequent decades, it was found that one has to contend with some
 infinitely generated modules, even if one is only interested in properties of the finitely generated ones. For example, in~\cite{Benson/Carlson/Rickard:1996a}, certain
infinitely generated idempotent modules are used in order to
classify the thick tensor ideals of the stable category of finitely
generated modules in terms of the prime ideal spectrum of the
cohomology ring $H^*(G,k)$.

An analogy from homotopy theory is as follows. Suppose that all you are really interested in is finite CW complexes, or even just closed manifolds. You might start to develop cohomology theories such as ordinary cohomology, $K$-theory, or cobordism theory. You then find that these are representable, but only by infinite CW complexes, so you are forced to study these. Cohomology operations then correspond to maps between these representing objects, and this is what makes calculation possible.

In modular representation theory, the idempotent modules above are representing objects in a similar way, and eventually these led to a classification~\cite{Benson/Iyengar/Krause:2011b}
of the localising tensor ideals of the stable category $\StMod kG$ of all $kG$-modules. There are local cohomology functors $\gam_\fp$ corresponding to the homogeneous prime ideals
of $H^*(G,k)$, that pick out the minimal localising tensor ideals $\gam_\fp\StMod kG$, and all localising tensor ideals are determined by the minimal ones that they contain. This sets up a bijection with the subsets of the homogeneous primes in $H^*(G,k)$ that are different from the maximal ideal of positive degree elements.

The problem with infinitely generated modules is that some rather
unsettling things happen. For example, as soon as the Sylow
$p$-subgroups of $G$ are non-cyclic, there are $kG$-modules $M$
with $M\cong M\oplus M\oplus M$ but $M\not\cong M\oplus M$.
So cancellation is violated, and furthermore,
$M$ and $M\oplus M$ are examples of non-isomorphic modules,
each of which is isomorphic to a direct summand of the other.
This and further pathologies are described in
Brenner and Ringel~\cite{Brenner/Ringel:1976a}.
The idempotent modules appearing above do not participate in
such pathologies, and so it is natural to ask for more comfortable
contexts containing these, and sitting between the finitely generated
modules and all modules.

Our goal in this paper is to describe a thick subcategory of $\gam_\fp\StMod kG$ that is large enough to contain the modules used in various constructions, but small enough that pathological behaviour is avoided. The simplest description is that this consists of objects $M$ \emph{of finite length} in the tensor triangulated category $\cat T=\gam_\fp \StMod kG$. This condition means that for each compact object $C$ in $\cat T$ the graded morphisms $\Hom^*_{\cat T}(C,M)$ form a module of finite length over the graded endomorphism ring of the tensor unit. In particular, $M$ is endofinite in the sense of \cite{Krause/Reichenbach:2001a}. This subcategory has many descriptions, and here we give some of these. 

\begin{theorem}
\label{th:finitelength}
Fix $\fp$ in $\Proj H^*(G,k)$ and a $\pi$-point $\alpha\colon K[t]/(t^p)\to KG$ representing $\fp$. For any object $M$ in the tensor triangulated category $\gam_\fp\StMod kG$, the  conditions below are equivalent.
\begin{enumerate}[\quad\rm(1)]
    \item\label{it:finlength} 
    $M$ has finite length in $\gam_\fp\StMod kG$;
    \item\label{it:dual} 
    $M$ is dualisable in $\gam_\fp\StMod kG$;
      \item\label{it:thick} 
    $M$ is in the thick subcategory generated by all $\gam_\fp C$ for $C\in\stmod kG$;
     \item\label{it:pi} 
    $\alpha^*(K\otimes_k M)$ is stably finite dimensional.
\end{enumerate}
\end{theorem}

There is a one-to-one correspondence between non-maximal homogeneous prime ideals in $H^*(G,k)$ and equivalence classes of $\pi$-points $\alpha$; see~\cite{Friedlander/Pevtsova:2007a}. Condition~\eqref{it:pi} above means that the restriction of the $KG$-module $K\otimes_kM$ to $K[t]/(t^p)$ is a direct sum of a finitely generated module and a projective module. 

The compact objects in $\gam_\fp\StMod kG$ have finite length, but the converse only holds when the prime $\fp$ is minimal.
In particular, the tensor identity $\gam_\fp k$ is dualisable, but compact only when $\fp$ is minimal. So given the result above  we may think of the category of finite length objects as the tensor triangulated completion of the compact objects. 

The equivalence \eqref{it:finlength}$\Leftrightarrow$\eqref{it:pi} is verified in Section~\ref{se:pipoints}. It applies equally to general finite group schemes and not only to finite groups. Condition \eqref{it:finlength} is equivalent to a noetherian property of the $\fp$-completion of $M$. This translation requires a version of Matlis duality for tensor triangulated categories, developed in Section~\ref{se:bc-duality}. Carlson~\cite{Carlson:2024a} proves that this is equivalent to condition \eqref{it:pi} when $\fp$ is a closed point and $G$ is elementary abelian. 

The implications \eqref{it:thick}$\Rightarrow$\eqref{it:dual}$\Rightarrow$\eqref{it:finlength} hold in a broader context; see Sections~\ref{se:bc-duality} and \ref{se:regularity}. They have been established in stable homotopy theory by Hovey and Strickland~\cite{Hovey/Strickland:1999a}. As explained in their work, the implication \eqref{it:finlength}$\Rightarrow$\eqref{it:thick} does not hold there; see also \cite{Benson/Iyengar/Krause/Pevtsova:2024b}.  The key property of $\StMod kG$ that is exploited in our proof of this implication is expressed in the following result.

\begin{theorem}
\label{th:local-regularity}
    For each $\fp$ in $\Proj H^*(G,k)$, the subcategory 
    \[
    \thick\{\gam_\fp C\mid C\in \stmod kG\}
    \]
    of $\StMod kG$ has a strong generator.
\end{theorem}

The notion of a strong generator of a category is due to Bondal and van den Bergh~\cite{Bondal/vandenBergh:2003a}. For the derived category  of a commutative noetherian ring $A$ and a prime ideal $\fp\subseteq A$, the analogue of the above category is $\thick\{\gam_\fp P\mid P \text{ perfect}\}$, and this has a strong generator precisely when the local ring $A_\fp$ is regular. One can thus view the theorem above as stating that $\StMod kG$ is \emph{locally regular}.  Establishing this property of $\StMod kG$ in its various incarnations and exploring its ramifications forms a major part of this work. The regularity result itself appears in Theorem~\ref{th:regularity-main}; see also the discussion following that result.
The proof is quite involved because there are, in general, many more objects in the thick closure of $\gam_\fp(\stmod kG)$ than in its idempotent completion. Examples are given in \cite{Carlson/Iyengar:2024a}. 

We give many other equivalent descriptions of dualisable objects of $\gam_\fp\StMod kG$ in Section~\ref{se:finite-groups} based on the ground work in Sections~\ref{se:cohomology-and-localisation}-\ref{se:regularity}, leading to many pleasing properties of these modules. One of them is that the Krull--Remak--Schmidt Theorem holds, so there are no pathologies of the type mentioned earlier. Another is that dualisability is preserved by the appropriate versions of induction and restriction. In particular, this lays the foundation  for the theory of vertices and sources, and makes the Green correspondence work the way one would desire. There is a version for $\StMod kG$, but it is in terms of equivalences of quotient categories rather than inducing and restricting indecomposable objects. The Burry--Carlson--Puig Theorem also works, as does Green's Indecomposability Theorem; this fails in $\StMod kG$~\cite{Benson/Wheeler:1999a}.

Here is a brief outline of this paper: Sections~\ref{se:cohomology-and-localisation}-\ref{se:dualisability} provide a summary of  basic facts about compactly generated triangulated categories with additional structure given by a linear action of a ring or a tensor product. Finite length objects are discussed in Section~\ref{se:finite-length}. Brown--Comenetz duality, which may be viewed as the tensor triangular analogue of Matlis duality from commutative algebra, is the subject of Section~\ref{se:bc-duality}. Sections~\ref{se:strong-gen}-\ref{se:regularity} are devoted to a notion of local regularity and various descriptions of local dualisable objects. The final Sections~\ref{se:finite-groups}-\ref{se:pipoints} contain our results on the modular representations of finite groups. As is clear from this description, a major part of this manuscript develops general machinery that may be of interest beyond modular representations.

\section{Rappels: cohomology and localisation}

\label{se:cohomology-and-localisation}

In this section we recall certain basic notions and constructions in
triangulated categories. Our main interest is in tensor triangulated
categories, but many of the fundamental constructions do not require
this additional structure, and are more transparent when treated in
greater generality. Primary references for the material presented here
are \cite{Benson/Iyengar/Krause:2008a, Benson/Iyengar/Krause:2012b}.

Throughout we fix a graded commutative noetherian ring $R$.  We  consider only homogeneous elements and ideals in $R$ and $R$-modules are usually graded. Morphisms between graded modules are degree preserving, but we suppress shifts of graded modules when they arise. In this spirit, `localisation' means homogeneous localisation, and $\Spec R$ denotes the set of homogeneous prime ideals in $R$. In particular, for $\fp\in \Spec R$ we write $R_\fp$ for the homogeneous localisation of $R$ at $\fp$; it is a graded local ring with unique homogeneous maximal ideal $\fp R_\fp$. 

Fix a compactly generated triangulated category ${\cat T}$; its full subcategory consisting of compact objects is denoted $\comp{\cat T}$. Fix as well an $R$-linear action on ${\cat T}$, in the sense defined below.

\subsection*{Central actions} 
For objects $X$ and $Y$ in ${\cat T}$ set
\[ 
\Hom_{\cat T}^*(X,Y)\coloneqq \bigoplus_{i\in\bbZ}\Hom_{\cat T}(X,\Si^i Y) 
\quad \text{and}\quad
\End_{\cat T}^{*}(X)\coloneqq  \Hom_{\cat T}^{*}(X,X)\,.
\] 
Composition makes $\End_{\cat T}^{*}(X)$ a graded ring and $\Hom_{\cat T}^{*}(X,Y)$ a left-$\End_{\cat T}^{*}(Y)$ right-$\End_{\cat T}^{*}(X)$ module. The category ${\cat T}$ is said to be \emph{$R$-linear} if for each $X$ in ${\cat T}$ there is a homomorphism of graded rings $\phi_X\colon R\to \End_{\cat T}^{*}(X)$ such that the induced left and right actions of $R$ on $\Hom_{\cat T}^{*}(X,Y)$ are compatible in the following sense: For any $r\in R$ and $\alpha\in\Hom^*_{\cat T}(X,Y)$, one
has $\phi_Y(r)\alpha=(-1)^{|r||\alpha|}\alpha\phi_X(r)$.

The $R$-linear category $\cat T$ is said to be \emph{noetherian} if the $R$-module $\Hom^*_{\cat T}(X,Y)$ is noetherian for all $X,Y$ in $\comp{\cat T}$.

\subsection*{Localisation and torsion}
Let $V\subseteq\Spec R$ be a \emph{specialisation closed} subset, so  $\fp\subseteq\fq$ and $\fp\in V$ implies $\fq\in V$. An $R$-module $M$ is \emph{$V$-torsion} if $M_\fp=0$ for all $\fp$ in $(\Spec R)\setminus V$. Given an ideal $\fa$ in $R$ we sometimes speak of $\fa$-torsion modules instead of $V(\fa)$-torsion modules.

An object $X$ in ${\cat T}$ is \emph{$V$-torsion} if the $R$-module $\Hom_{\cat T}^*(C,X)$ is $V$-torsion for all $C\in\comp{\cat T}$. An object $Y$ in ${\cat T}$ is \emph{$V$-local} if $\Hom_{\cat T}(X,Y)=0$ for all $V$-torsion objects $X\in{\cat T}$. Writing $\cat S$ for the full subcategory of $V$-local objects in ${\cat T}$, one has a recollement of triangulated categories
\begin{equation}
\label{eq:recoll-spcl}
\begin{tikzcd}[column sep = huge]
    {\cat S} \arrow[tail]{r}[description]{i} &{\cat T}\arrow[twoheadrightarrow,swap,yshift=1.5ex]{l}{i_\lambda}
    \arrow[twoheadrightarrow,yshift=-1.5ex]{l}{i_\rho}
    \arrow[twoheadrightarrow]{r}[description]{p} &{\cat T}/{\cat S}
    \arrow[tail,swap,yshift=1.5ex]{l}{p_\lambda}\arrow[tail,yshift=-1.5ex]{l}{p_\rho}
\end{tikzcd}
\end{equation}
by \cite[Proposition~2.3]{Benson/Iyengar/Krause:2012b}. The exact functors
\[ 
L_V\coloneqq ii_\lambda \qquad\text{and}\qquad\gam_V\coloneqq p_\lambda p
\] 
are called \emph{localisation} and \emph{local cohomology} with respect to $V$. They provide, for each object $X\in{\cat T}$, a functorial exact triangle
\begin{equation}
\label{eq:lch}
\gam_VX\lra X\lra L_V X\lra\,.
\end{equation}
An object $X\in\cat T$ is $V$-torsion if and only if the counit $\gam_V X\to X$ is an isomorphism, and $X$ is $V$-local if and only if the unit $X\to L_V X$ is an isomorphism. In particular, the image of $\gam_V$ equals the subcategory of $V$-torsion objects in ${\cat T}$, while the image of $L_V$ equals the subcategory of $V$-local objects in ${\cat T}$. It is clear from the definition that $\gam_V \cat T$ and $L_V\cat T$ are triangulated subcategories of $\cat T$.

\begin{lemma}
\label{le:gam-loc-V}
    The triangulated categories $\gam_V\cat T$ and $L_V\cat T$ have coproducts and are compactly generated. We have $\comp{(\gam_V\cat T)}=\comp{\cat T}\cap\gam_V\cat T$, while $\comp{(L_V\cat T)}$ and $L_V(\comp{\cat T})$ coincide up to direct summands.
\end{lemma}
\begin{proof}
     It follows from \cite[Theorem~6.4]{Benson/Iyengar/Krause:2008a} that $\gam_V{\cat T}$ is compactly generated by the compact objects from $\cat T$ that are $V$-torsion. The claims concerning $L_V\cat T$ are contained in  \cite[Theorem~2.1]{Neeman:1992b}.
\end{proof}

\subsection*{Completion}
The \emph{completion} with respect to $V$ is the functor
\[
\lam^V\coloneqq p_\rho p\,,
\] 
and an object $X\in{\cat T}$ is \emph{$V$-complete} if the unit $X\to \lam^VX$ is an isomorphism.

The following result describes the interplay between torsion and completion. This phenomenon was discovered by Greenlees and May~\cite{Greenlees/May:1992a}, and one finds many similar statements in the literature; see for example~\cite[Theorem~2.1]{Dwyer/Greenlees:2002a} or \cite[Theorem~3.3.5]{Hovey/Palmieri/Strickland:1997a}, and also \cite[Proposition~2.3]{Benson/Iyengar/Krause:2012b}. 

\begin{proposition} 
\label{pr:tor-complete}
For  each $X$ in ${\cat T}$ the canonical maps $\gam_V X\to X$ and  $X\to \lam^V X$ induce isomorphisms
\[
\lam^V\gam_V X\longiso \lam^V X \qquad\text{and}\qquad \gam_V X\longiso \gam_V\lam^V X\,.
\]
Moreover, the functors $(\gam_V,\lam^V)$ form an adjoint pair  and induce mutually quasi-inverse triangle equivalences
\[
\begin{tikzcd}[column sep = huge]
    \gam_V{\cat T}  \arrow[swap,yshift=-.75ex]{r}{\lam^V}  &\lam^V{\cat T}\arrow[swap,yshift=.75ex]{l}{\gam_V}\,.
    \end{tikzcd}
 \]
\end{proposition}

\begin{proof}
Since $\lam^V L_V=p_\rho pii_\lambda =0$, applying $\lam^V$ to the exact triangle \eqref{eq:lch} yields an isomorphism $\lam^V\gam_V X\iso \lam^V X$. An analogous argument yields $\gam_V X\iso \gam_V\lam^V X$.

For $X,Y\in{\cat T}$ we have isomorphisms
\[
\Hom_{\cat T}(p_\lambda pX,Y)\cong \Hom_{\cat T/\cat S}(pX,pY) \cong  \Hom_{\cat T}(X,p_\rho pY)
\]
and hence $(\gam_V,\lam^V)$ form an adjoint pair. The first part of the lemma then yields the equivalences $\gam_V{\cat T}\iso \lam^V{\cat T}$
and $\lam^V{\cat T}\iso \gam_V{\cat T}$. 
\end{proof}

\subsection*{Localisation and completion at a point}
Fix a point $\fp\in\Spec R$ and consider the specialisation closed subsets
\[ 
V(\fp)\coloneqq \{\fq\in\Spec R\mid \fp\subseteq\fq\}\qquad\text{and}\qquad
  Z(\fp)\coloneqq \{\fq\in\Spec R\mid \fq\not\subseteq\fp\}\,.
\]
For $X\in\cat T$ the \emph{localisation} of $X$ at $\fp$ is
\[
X_\fp\coloneqq L_{Z(\fp)}X\,.
\] 
The object $X$ is said to be \emph{$\fp$-local} if the natural localisation map is an isomorphism: $X\iso
X_\fp$. Set
\[ 
{\cat T}_\fp\coloneqq \{X\in{\cat T}\mid X \text{ is $\fp$-local}\}
\] 
for the full subcategory of $\fp$-local objects in $\cat T$. For all $C\in\comp{\cat T}$ the natural morphism $X\to X_\fp$ induces an isomorphism 
\[
\Hom_{\cat T}^*(C,X)_\fp\iso \Hom_{\cat T}^*(C,X_\fp)\,.
\] 
Thus an object $X$ is $\fp$-local if and only if $\Hom_{\cat T}^*(C,X)$ is $\fp$-local for all $C\in\comp{\cat T}$. For $X,Y$ in ${\cat T}_\fp$ the $R$-module $\Hom_{\cat T}^*(X,Y)$ is $\fp$-local, so the action of $R$ on ${\cat T}_\fp$ factors through $R_\fp$. In particular, ${\cat T}_\fp$ is an $R_\fp$-linear category. The following result is used frequently; the proof is clear.

\begin{lemma}
\label{le:local-noetherian}
Let $\fp$ in $\Spec R$. For all $C\in\comp{\cat T}$ and $X\in \cat T$ there are isomorphisms of $R_\fp$-modules
\[
 \Hom^*_{\cat T}(C,X)_\fp \cong\Hom^*_{\cat T}(C,X_\fp)\cong\Hom^*_{\cat T}(C_\fp,X_\fp)\,.
\]
Thus, if the $R$-linear category $\cat T$ is noetherian, so is the $R_\fp$-linear category ${\cat T}_\fp$. \qed
\end{lemma}

We call the functor $\gam_{\fp}\colon {\cat T}\to\cat T$ given by $X\mapsto \gam_{V(\fp)}X_\fp$ the \emph{local cohomology functor} with respect to $\fp$.  The functor $\lam^{\fp}\colon \cat T\to \cat T$ given by $X\mapsto\lam^{V(\fp)} ii_{\rho}X$, where the adjoint pair $(i,i_\rho)$ is from  \eqref{eq:recoll-spcl} for $V=Z(\fp)$, is right adjoint to $\gam_{\fp}$. It is easy to verify from the recollement defining these functors that the images of $\gam_\fp$ and
$\lam^\fp$ are the following subcategories:
\begin{align*}
\gam_\fp{\cat T} &\coloneqq \{X\in{\cat T}_\fp \mid X \text{ is $V(\fp)$-torsion} \} \\
\lam^\fp{\cat T} &\coloneqq \{X\in{\cat T}_\fp \mid X \text{ is $V(\fp)$-complete} \}. 
\end{align*}
We speak of objects in $\gam_{\fp}{\cat T}$ as being \emph{$\fp$-local
  $\fp$-torsion} and those in $\lam^{\fp}{\cat T}$ as being
\emph{$\fp$-local $\fp$-complete}. Proposition~\ref{pr:tor-complete} implies
that $\gam_\fp$ and $\lam^\fp$ restrict to mutually quasi-inverse triangle equivalences
\begin{equation}
\label{eq:DG}
  \begin{tikzcd}[column sep = huge]
\gam_\fp{\cat T}\arrow[rightarrow,yshift=-.75ex,swap,r,"\lam^\fp"]
	&\lam^\fp{\cat T}\,.\arrow[rightarrow,yshift=.75ex,l,swap,"\gam_\fp"]
\end{tikzcd} 
\end{equation}
A caveat is in order here: The equivalence above is not compatible with the functors from $\cat T$, in that, $\lam^\fp\gam_\fp\not\cong \lam^\fp$ and
$\gam_\fp\lam^\fp\not\cong \gam_\fp$ in general. However,  restricted to $\cat T_\fp$ they are isomorphisms, by Proposition~\ref{pr:tor-complete}. Moreover, in applications the more natural completion functor to consider is $\fp$-completion after localisation at $\fp$, akin to the $\fp$-torsion functor. All these considerations suggest restricting attention to $\fp$-local objects and then the following
perspective is useful: The inclusion of the full subcategory
\[
    \cat T_{<\fp}\coloneqq\{X\in\cat T_\fp\mid\gam_\fp X=0\}=
        \{X\in\cat T_\fp\mid\lam^\fp X=0\}
        \] of
$\fp$-local objects (co)supported away from $\fp$ yields a recollement 
\[
\begin{tikzcd}[column sep = huge]
    {\cat T}_{<\fp}\arrow[tail]{r} 
    	&{\cat T_\fp}\arrow[twoheadrightarrow,swap,yshift=1.5ex]{l}    \arrow[twoheadrightarrow,yshift=-1.5ex]{l}
		    \arrow[twoheadrightarrow]{r}[description]{q} 
        &{\cat T_\fp}/{\cat T}_{<\fp}    \arrow[tail,swap,yshift=1.5ex]{l}{q_\lambda}\arrow[tail,yshift=-1.5ex]{l}{q_\rho}
  \end{tikzcd}
\]
of
triangulated categories. It is a special case of \eqref{eq:recoll-spcl}. Then $\gam_\fp\cong q_\lambda q$ and $\lam^\fp\cong q_\rho q$.

\subsection*{Koszul construction} 
 Given an element $r$ in $R$ and an object $C$ in $\cat T$, we write $\kos C{r}$ for a cone of the morphism $C\xra{r} \Si^{|r|}C$. Given a sequence of elements $\bsr\coloneqq r_1,\dots,r_n$ in $R$, we write $\kos C{\bsr}$ for any object obtained by taking iterated cones; this is the \emph{Koszul construction} on $C$ with respect  to $\bsr$; see \cite[Definition~5.10]{Benson/Iyengar/Krause:2008a}. For an ideal $\fa$ of $R$ we write $\kos{C}{\fa}$ for any Koszul object $\kos C{\bsr}$ such that $\bsr$ generates $\fa$.

\begin{lemma}
\label{le:p-torsion}
Fix $\fp$ in $\Spec R$. There is an equality  $\comp{(\gam_\fp\cat T)}=\comp{(\lam^\fp\cat T)}$,
    and this thick subcategory is generated by objects of the form $\kos C{\fp}$ as $C$ ranges over the compact objects of $\cat T_\fp$. Hence for any compact object $C$ in $\gam_\fp\cat T$ or $\lam^\fp\cat T$, there exists an integer $n$ such that 
    \[
    \fp^n\cdot \Hom^*_{\cat T}(C,-) = 0= \fp^n\cdot \Hom^*_{\cat T}(-,C)\,.
    \]
\end{lemma}
 
 \begin{proof}
 The description of the compact objects in  $\gam_{\fp}\cat T$ follows from \cite[Theorem~6.4]{Benson/Iyengar/Krause:2008a}; see also \cite[Proposition~2.7]{Benson/Iyengar/Krause:2011a}. Cosupport theory shows that $\lam^\fp(\kos X{\fp})=\kos X{\fp}$ for any 
 object $X$ in $\cat T_\fp$; see \cite[\S4]{Benson/Iyengar/Krause:2012b} and in particular \cite[Lemma~4.12]{Benson/Iyengar/Krause:2012b}. This yields the 
 equality $\comp{(\gam_\fp\cat T)}=\comp{(\lam^\fp\cat T)}$, since $\lam^\fp$ induces an equivalence $\gam_\fp\cat T\iso\lam^\fp\cat T$.
 From \cite[Lemma~5.11]{Benson/Iyengar/Krause:2008a} it follows that $ \fp^n\cdot \End^*_{\cat T}(C) = 0$ for $n\gg 0$.
 This yields the second assertion for compacts in $\gam_\fp\cat T$ and $\lam^\fp\cat T$.
 \end{proof}

We also record the following observation, for later use. Given a graded $R$-module $M$, the object $\kos M{\fa}$ is the Koszul construction on $M$ with respect to the ideal $\fa$, where $M$ is viewed as an object in the derived category of graded $R$-modules. In particular, each $H^i(\kos M{\fa})$ is itself a graded $R$-module.

\begin{lemma}
\label{le:koszul-tt}
    For objects  $X,Y$ in $\cat T$ and an ideal $\fa\subseteq R$ one has
    \begin{align*}
     \length_R \Hom^*_{\cat T}(X,\kos Y\fa)< \infty &\iff 
        \length_R H^*(\kos{\Hom^*_{\cat T}(X,Y)}{\fa})< \infty\\
        &\iff \length_R \Hom^*_{\cat T}(\kos X\fa,Y)< \infty\,.   
    \end{align*}
\end{lemma}

\begin{proof}
We verify the first isomorphism; a similar argument gives the second one. It suffices to consider the case when $\fa=(r)$. Set $M=\Hom^*_{\cat T}(X,Y)$ and let $L$ be the submodule of elements annihilated by $r$. From the construction of the Koszul object $\kos Yr$ one gets an exact sequence of $R$-modules
\[
0\longrightarrow M/rM \longrightarrow \Hom^*_{\cat T}(X,\kos Yr) 
     \longrightarrow  L \longrightarrow 0\,.
\]
Thus $\Hom^*_{\cat T}(X,\kos Yr)$ has finite length if and only if both $M/rM$ and $L$ have finite length. 
For the Koszul object $\kos Mr$ we have
\[
H^i(\kos Mr)=
\begin{cases} L&i=-1\,,\\ M/rM&i=0\,,\\ 0&i\neq -1,0\,.
\end{cases}
\]
The stated equivalences are now obvious.
\end{proof}

\section{Tensor structure and dualisable objects}
\label{se:dualisability}

In this section we recall some basic facts about compactly generated tensor triangulated categories. We pay particular attention to the notion of dualisability; a more detailed discussion can be found in \cite{Benson/Iyengar/Krause/Pevtsova:2024b}. 

Fix a compactly generated tensor triangulated category $({\cat T},\otimes,\one)$, with tensor product $\otimes$ and unit $\one$. Throughout we assume the tensor product is symmetric. For now $\one$ need not be compact. Brown representability yields \emph{functions objects} $\fHom(X,Y)$ satisfying an adjunction isomorphism
\[
\Hom_{\cat T}(X\otimes Y,Z)\cong \Hom_{\cat T}(X,\fHom(Y,Z)) \quad \text{for all $X,Y,Z$ in ${\cat T}$.}
\]

The \emph{Spanier--Whitehead dual} of an object $X$ is
\[
\swd X \coloneqq \fHom(X,\one)\,.
\]
The assignment $X\mapsto \swd X$ is a contravariant functor ${\cat T}\to{\cat T}$.  

\subsection*{Dualisable objects}
An object $X$ in ${\cat T}$ is said to be \emph{dualisable} if for all
$Y$ in ${\cat T}$ the natural map
\[
\swd X\otimes Y\lra \fHom(X,Y)\,,
\]
is an isomorphism. 

The adjoint of the composite $X\otimes \swd X \xra{\gamma}\swd X\otimes X \xra{\eps}\one$ of the braiding $\gamma$ with the counit $\eps$  gives the natural double duality map
\begin{equation*}
\label{eq:rho}
\rho \colon X \lra \swd\swd(X)\,.
\end{equation*} 

The result below is \cite[Proposition~2.2]{Benson/Iyengar/Krause/Pevtsova:2024b} and collects some useful observations concerning these notions.

\begin{proposition}\pushQED{\qed}
\label{pr:dualisable}
Let $X$ be an object in ${\cat T}$. The following statements hold.
\begin{enumerate}[\quad\rm(1)]
\item
The object $X$ is dualisable if and only if there is a map $\eta\colon \one \to X\otimes \swd X$ making the following diagram commute
\[
\begin{tikzcd}
\one \arrow{d} \arrow["\eta"]{r} & X\otimes \swd X \arrow["\gamma"]{d} \\
\fHom(X,X) \arrow[leftarrow]{r} & \swd X\otimes X
\end{tikzcd}
\]
where the left hand map is the adjoint to the isomorphism $\one \otimes X \iso X$.
\item
If $X$ is dualisable so is $\swd X$ and $\rho\colon X\to \swd\swd X$ is an isomorphism.
\item
If $X$ is dualisable and $C\in{\cat T}$ is compact, then $C\otimes X$ is compact. \qedhere
\end{enumerate}
\end{proposition}

\subsection*{Central actions}

From this point on $({\cat T},\otimes,\one)$ is a rigidly compactly generated  tensor triangulated category, where \emph{rigidity} means that compact and dualisable objects in $\cat T$ coincide. Hence $\one$ is compact, $\comp{\cat T}\subseteq \cat T$ is a tensor ideal thick subcategory, and Spanier--Whitehead duality induces an auto-equivalence on it.

Let $R$ be a noetherian graded commutative ring that acts \emph{canonically} on $\cat T$, that is to say, via a map of graded rings $R\to \End_{\cat T}^{*}(\one)$.

We begin by recording some special features of the localisation and
completion functors, recalled in
Section~\ref{se:cohomology-and-localisation}, in this context.  For
any specialisation closed subset $V\subseteq \Spec R$ and object
$X \in{\cat T}$ there are natural isomorphisms
\begin{equation*}
\gam_V X\cong \gam_V \one \otimes X \qquad\text{and}\qquad \loc_V X \cong  \loc_V \one \otimes X \,.
\end{equation*}
By adjunction, these yield for any $Y\in \cat T$ natural isomorphisms
\begin{equation*}
\fHom(\gam_V X,Y)\cong\lam^V\fHom(X,Y) \cong\fHom(X,\lam^V Y)\,.
\end{equation*}
See \cite[Proposition~8.1]{Benson/Iyengar/Krause:2008a} and
\cite[Proposition~8.3]{Benson/Iyengar/Krause:2012b}.  
In particular, for any $\fp\in\Spec R$ and objects $X,Y\in\cat T_\fp$ we have
\begin{equation}\label{eq:loc-hom}
\fHom(\gam_\fp X,Y)\cong\lam^\fp\fHom(X,Y) \cong\fHom(X,\lam^\fp Y)\,.
\end{equation}

\subsection*{Tensor structures of subcategories}
We  describe the tensor structures of the subcategories $L_V\cat T$ and
$\gam_V\cat T$ that are given by a specialisation closed subset
$V\subseteq \Spec R$. In particular we study their dualisable objects.

The following lemma is about the shadow of Spanier--Whitehead duality in $\loc_V{\cat T}$,
the subcategory of ${\cat T}$ consisting of $V$-local objects.

\begin{lemma}
\label{le:LV-cat}
For each  specialisation closed subset $V\subseteq \Spec R$ one has  a commutative diagram of exact  functors:
\[
  \begin{tikzcd}
    \comp{\cat T} \arrow[d,swap,"{\fHom(-,\one)}"]\arrow[rr,"\loc_V"]&&
    \loc_V{\cat T} \arrow[d,"{\fHom(-,\loc_V\one)}"]\\
   \comp{\cat T}\arrow[rr,"\loc_V"]&&\loc_V{\cat T}
 \end{tikzcd}
\]
\end{lemma}

 \begin{proof}
Fix a compact object $C$ in ${\cat T}$. The following series of isomorphisms verifies the commutativity of the diagram:
  \begin{align*}
 \fHom(\loc_V\one \otimes C,\loc_V\one) 
 	&\cong    \fHom(\loc_V\one,\fHom(C,\loc_V\one)) \\
	&\cong    \fHom(\loc_V\one,\swd{C}\otimes\loc_V\one) \\
	&\cong    \fHom(\loc_V\one, \loc_V\swd C) \\
    &\cong    \fHom(\one, \loc_V\swd C) \\
	&\cong  \loc_V\swd C\,.
\end{align*}
The second one holds because compact objects are dualisable, by rigidity of $\cat T$.
\end{proof}

 The result below builds on Lemma~\ref{le:gam-loc-V}.
 
 \begin{proposition}
 \label{pr:loc-V}
 Let $V\subseteq \Spec R$ be a specialisation closed subset. The category $\loc_V{\cat T}$ is compactly generated and inherits a tensor product from $\cat T$. The unit is $\loc_V\one$ and the function object is the same as for ${\cat T}$. The compact objects and dualisable objects in $\loc_V{\cat T}$  coincide and are the  direct summands of objects in $\loc_V(\comp{\cat T})$.  
 \end{proposition}
 
 \begin{proof}
 The claim about compact generation and the description of compact objects is from  Lemma~\ref{le:gam-loc-V}. For $X,Y$ in $\cat T$ the  maps $X\to \loc_VX$ and $Y\to\loc_VY$ induce isomorphisms
 \[
 \loc_V(X\otimes Y)  \longiso\loc_VX\otimes Y \longiso  \loc_V X\otimes \loc_VY\,.
 \]
 It follows that when $X,Y$ are $V$-local, so is $X\otimes Y$, and
 hence that the tensor product on $\cat T$ restricts to a tensor
 product on $\loc_V\cat T$. Setting $Y=\one$ one gets also that
 $\loc_V\one$ is the unit in $\loc_V\cat T$.  The isomorphism
 \[\Hom_{\cat T}(-, \fHom(X,Y))\cong \Hom_{\cat T}(-\otimes X,Y)\]
 shows that $\fHom(X,Y)$ is $V$-local when $Y$ is $V$-local. Thus the
 function object in $\loc_V{\cat T}$ is the same as the one in
 $\cat T$.

 Since $\loc_V\one$ is compact in $\loc_V\cat T$, dualisable objects are
 compact, by Proposition~\ref{pr:dualisable}. As to the converse, it
 suffices to verify that objects of the form $\loc_VC$, with $C$ in
 $\comp{\cat T}$, are dualisable in $\loc_V{\cat T}$.  For any $Y$ in
 ${\loc_V\cat T}$ one has isomorphisms
\begin{align*}
\fHom(\loc_VC,\loc_V\one)\otimes Y
	& \cong \loc_V \fHom(C,\one) \otimes Y \\
	& \cong  \fHom(C,\one) \otimes Y \\  
	& \cong \fHom(C,Y) \\
	& \cong \fHom(\loc_VC,Y)
\end{align*}
where the first isomorphism is by the commutativity  of the square in Lemma~\ref{le:LV-cat}.  Thus $\loc_VC$ is dualisable. 
 \end{proof}
 
 Next we discuss the analogues of the preceding results for
 $\gam_V{\cat T}$, the subcategory of $\cat T$ consisting of
 $V$-torsion objects.

\begin{lemma}
\label{le:gamV-cat}
For each  specialisation closed subset $V\subseteq \Spec R$ one has commutative diagrams of exact  functors:
\[
  \begin{tikzcd}
    \comp{\cat T} \arrow[d,swap,"{\fHom(-,\one)}"]\arrow[rr,"\gam_V"]&&    \gam_V{\cat T} \arrow[d,"{\fHom(-,\gam_V\one)}"]\\
   \comp{\cat T}\arrow[rr,"\lam^V"]&&\lam^V{\cat T}
 \end{tikzcd}
 \quad 
  \begin{tikzcd}
    \comp{\cat T}\arrow[d,swap,"{\fHom(-,\one)}"]\arrow[rr,"\lam^V"] &&    \lam^V{\cat T} \arrow[d,"{\gam_V\fHom(-,\lam^V\one)}"]\\
   \comp{\cat T}\arrow[rr,"\gam_V"]&&\gam_V{\cat T}
 \end{tikzcd}
\]
\end{lemma}

\begin{proof}
Fix a compact object $C$ in $\cat T$. The commutativity of the square on the left is verified below:
  \[
  \fHom(\gam_V C,\gam_V\one)\cong    \fHom(C,\lam^V\one)\cong\lam^V\fHom(C,\one)\,.
\]
Next we verify the commutativity of the square on the right:
\begin{align*}
  \gam_V\fHom(\lam^V C,\lam^V\one)
  &\cong \gam_V\fHom(C,\lam^V \one)\\
  &\cong \gam_V\lam^V\fHom(C,\one)\\
  &\cong \gam_V\fHom(C,\one)\,.\qedhere
\end{align*}  
\end{proof}

It follows from the result below that, typically, $\gam_{V}{\cat T}$
has more dualisable objects than compact ones.

 \begin{proposition}
 \label{pr:gam-V}
 Let $V\subseteq \Spec R$ be a specialisation closed subset. The category $\gam_V{\cat T}$ is 
 compactly generated and inherits a tensor product from $\cat T$. The unit is $\gam_V\one$ and
 the function object is $\gam_V\fHom(-,-)$. The compact objects in $\gam_V\cat T$ are
 \[
\comp{(\gam_V\cat T)} = \comp{\cat T}\cap \gam_V\cat T\,,
\]
and these are dualisable. All objects in $\gam_V{(\comp{\cat T})}$ are dualisable.
 \end{proposition}

 \begin{proof}
 Compact generation and description of compact objects is from  Lemma~\ref{le:gam-loc-V}. From the description of $\gam_V$ as a tensor functor it is clear that the tensor product on
   $\cat T$ restricts to a tensor product on $\gam_V{\cat T}$. The
   functor $\gam_V$ provides a right adjoint for the inclusion
   $\gam_V{\cat T}\to\cat T$, and therefore it preserves function objects.
   
It remains to verify that $\gam_VC$ is dualisable in $\gam_V{\cat T}$ for any compact object $C$ in $\cat T$.    For any $Y$ in $\gam_V{\cat T}$ one gets
\begin{align*}
  \gam_V\fHom(\gam_V C,\gam_V\one) \otimes Y
  &\cong   \gam_V\lam^V\swd{(C)}\otimes Y\\
  &\cong \gam_V\swd{(C)}\otimes Y \\
  &\cong \gam_V(\swd{(C)}\otimes Y) \\
  &\cong   \gam_V \fHom(C,Y)\\
  &\cong   \gam_V\fHom(\gam_V C,Y)
\end{align*}
where  the first isomorphism follows from the commutativity of the second square in   Lemma~\ref{le:gamV-cat}.
Thus $\gam_VC$ is dualisable.
\end{proof}

Next we focus on $\gam_{\fp}{\cat T}$, the subcategory of $\cat T$ consisting of $\fp$-local 
$\fp$-torsion objects, for some $\fp$ in $\Spec R$. The Spanier--Whitehead dual in the category 
$\gam_\fp \cat T$ is given by the functor
\[
\gam_\fp \fHom(-,\one_\fp) = \gam_\fp \fHom(-,\gam_\fp \one)\,.
\]
Since $\gam_{\fp}{\cat T}=\gam_{V(\fp)}{\cat T}\cap L_{Z(\fp)}\cat T$, from Propositions~\ref{pr:loc-V} and \ref{pr:gam-V} we get the result below.

\begin{corollary}
\label{co:plocal-cat}
For each $\fp$ in $\Spec R$ the category $\gam_{\fp}{\cat T}$ is
compactly generated and inherits a tensor product from $\cat T$.  The
unit is $\gam_\fp\one$ and the function object is
$\gam_\fp\fHom(-,-)$. The compact objects in $\gam_{\fp}\cat T$ are
\[
\comp{(\gam_{\fp}\cat T)} =  \comp{(\cat T_\fp)}\cap \gam_{\fp}\cat T\,.
\]
All objects in  $\gam_\fp(\comp{\cat T})$ are dualisable. \qed
 \end{corollary}

Here is the analogous result for $\lam^\fp\cat T$.

\begin{corollary}
\label{co:plocal-cat-lambda}
For each $\fp$ in $\Spec R$ the category $\lam^{\fp}\cat T$ is compactly generated and tensor triangulated, with product $\lam^\fp(-\otimes-)$.  The unit is $\lam^\fp(\one_\fp)$ and the function object is $\fHom(-,-)$, the function object on $\cat T$. The compact objects in $\lam^{\fp}\cat T$ are
\[
\comp{(\lam^{\fp}\cat T)} =  \comp{(\cat T_\fp)}\cap \gam_{\fp}\cat T=\comp{(\gam_{\fp}\cat T)} \,.
\]
All objects in  $\lam^\fp(\comp{\cat T}_\fp)$ are dualisable. \qed
 \end{corollary}

\section{Finite length objects}
\label{se:finite-length}
Let $R$ be a graded commutative noetherian  ring and $\cat T$ a compactly generated $R$-linear triangulated category; it need not be tensor triangulated. In what follows we consider various types of finiteness conditions for objects in $\cat T$ and its various subcategories. 
Here are the relevant concepts.

\begin{definition}
\label{de:finiteness}
    An object $X$ in $\cat T$ is said to be 
\begin{enumerate}[\quad\rm(1)]
\item  \emph{noetherian} over $R$ if the $R$-module $\Hom^*_{\cat T}(C,X)$ is noetherian for all $C\in\comp{\cat T}$,
\item  \emph{artinian} over $R$ if the $R$-module $\Hom^*_{\cat T}(C,X)$ is artinian for all $C\in\comp{\cat T}$,
\item  \emph{of finite length} over $R$ if $\Hom^*_{\cat T}(C,X)$ is of finite length over $R$ for all $C\in\comp{\cat T}$.
\end{enumerate}
\end{definition}

For tensor triangulated categories there is a more intrinsic notion of a finite length object.

\begin{remark}
\label{re:finite-length}
    Let $\cat T=(\cat T,\otimes,\one)$ be tensor triangulated and suppose it admits an $R$-action via a map of graded rings $R\to S=\End_{\cat T}^*(\one)$. Then we say that an object $X$ in $\cat T$ is \emph{of finite length} if $\Hom^*_{\cat T}(C,X)$ is of finite length over $S$ for all $C\in\comp{\cat T}$. This condition coincides with the one over $R$ when, for example, $S$ is finite over $R$ or when $R$ is local and $S$ is its completion at the maximal ideal.
\end{remark}

\begin{remark}
\label{re:degree-zero}
The above finiteness conditions on graded morphisms imply the analogous conditions for the degree zero morphisms, as a module over $R^0$, because for any objects $X,Y$ in $\cat T$ and  $M=\Hom^*_{\cat T}(X,Y)$, the poset of submodules of the $R^0$-module $M^0=\Hom_{\cat T}(X,Y)$ embeds into the poset of submodules of the graded $R$-module $M$. The map sends $U\subseteq M^0$ to $RU$, and its inverse sends $V\subseteq M$ to $V^0$.
\end{remark}

We proceed with a characterisation of the finite length objects in $\gam_\fp \cat T$. 
 
\begin{proposition}
    \label{pr:artinian}
For  $\fp\in\Spec R$ and $X\in\cat T_\fp$ the conditions below are equivalent.
\begin{enumerate}[\quad\rm(1)]
\item $X$ is in $\gam_{\fp}\cat T$ and of finite length over $R_\fp$;
\item $X$ is artinian  over $R_\fp$;
\item $\Hom^*_{\cat T}(C,X)$ is artinian over $R_\fp$ for each compact object $C$ in ${\cat T}$.
\end{enumerate}
\end{proposition}

\begin{proof}
By Lemmas~\ref{le:gam-loc-V} and \ref{le:local-noetherian} we can pass to $\cat T_\fp$ and $R_\fp$ and  suppose $R$ is local with maximal ideal  $\fp$. In particular,  (2) and (3) are equivalent.

  (1)$\Rightarrow$(2):  Let $C$ be a compact object in $\cat T$ and fix
  a sequence of elements $\bsr\coloneqq r_1,\ldots,r_n$ which
  generates the ideal $\fp$. Then $\kos C{\bsr}$ is $\fp$-torsion and
  hence in $\comp{(\gam_\fp\cat T)}$. The hypothesis yields that the
  $R$-module $\Hom_{\cat T}^*(\kos C{\bsr},X)$ has finite length, and
  in particular it is artinian. Since $\kos C{\bsr}$ is obtained from
  $C$ by an iterated Koszul construction, it suffices to prove that
  for any $r\in \fp$ and compact object $C$ in $\cat T$, if the
  $R$-module $\Hom_{\cat T}^*(\kos Cr,X)$ is artinian, then so is
  $\Hom_{\cat T}^*(C,X)$.

By definition of $\kos Cr$  one has an exact triangle  $C\xra{r} \Sigma^{|r|} C\to\kos{C}{r}\to $ and that induces an  exact sequence of $R$-modules
\[
\Hom_{\cat T}^*(\kos{C}{r},X)\lra \Hom_{\cat T}^{*-|r|}(C,X)\xra{\ r\ }  \Hom_{\cat T}^*(C,X)\,.
  \] 
The module $\Hom_{\cat T}^*(C,X)$ is $\fp$-torsion, since $X$ is in $\gam_{\fp}\cat T$.  As $\Hom_{\cat T}^*(\kos{C}{r},X)$ is artinian,  Lemma~\ref{le:coNak} below yields that $\Hom_{\cat T}^*(C,X)$ is also artinian.
  
  (2)$\Rightarrow$(1): Artinian $R$-modules are $\fp$-torsion, so the hypothesis implies $X$ is in $\gam_{\fp}\cat T$.  Any compact object $C$ in $\gam_{\fp}\cat T$ is also compact in $\cat T$, so the hypothesis implies that the $R$-module $\Hom^*_{\cat T}(C,X)$ is artinian; it is also annihilated by a power of $\fp$, by Lemma~\ref{le:p-torsion}.  Thus it has finite length.
\end{proof}

The  following observation was used in the argument above.

\begin{lemma}
\label{le:coNak}
Let $R$ be a local ring with maximal ideal $\fm$ and $M$ an
$\fm$-torsion $R$-module.  If for an $r\in\fm$ the $R$-module
$ \Ker(M\xra{r} M)$ is artinian, then so is $M$.
\end{lemma}

\begin{proof}
  Since $M$ is $\fm$-torsion its injective hull is a direct sum of copies of twists of $I(\fm)$, the number of copies being equal to the length of $\Soc M$, the socle of $M$. The module $I(\fm)$ is
  artinian by Matlis duality, so it suffices to verify that $\Soc M$ is finitely generated.

Set $L\coloneqq \Ker(M\xra{r} M)$ and consider the exact sequence of $R$-modules
\[
0\lra L\lra M \xrightarrow{\ r\ } M\,.
\]
Applying $\Hom^*_R(R/\fm,-)$ to it yields the exact sequence 
\[
0\lra \Soc L\lra \Soc M \xrightarrow{\ r\ } \Soc M
\]
and the right hand map is zero. Thus $\Soc M$ is isomorphic to $\Soc L$ and therefore finitely
generated, because $L$ is artinian.
\end{proof}

\subsection*{Minimal primes}
Let $\supp_R\cat T$ denote the \emph{support} of $\cat T$, the set of points $\fp$ in $\Spec R$ such that $\gam_\fp X\ne 0$ for some $X\in \cat T$.  
The result below augments Proposition~\ref{pr:artinian}.

\begin{lemma}
\label{le:minimal}
Suppose $\cat T$ is noetherian and fix $\fp$  in $\supp_R\cat T$. Then $\fp$ is minimal in $\supp_R{\cat T}$ if and only if in  $\cat T_\fp$  artinian and finite length objects over $R_\fp$ coincide.
\end{lemma}

\begin{proof}
First observe that $\fp\in\supp_R\cat T$ is minimal if and only if $\gam_\fp\cat T=\cat T_\fp$. From this it follows with Proposition~\ref{pr:artinian} that artinian and finite length coincide for the $R_\fp$-module $\Hom_{\cat T}^*(C,X)$ when $\fp$ is minimal.

As to the converse, we use the object $T_\fp C$ in $\gam_\fp\cat T$ which is defined below via \eqref{eq:inj-coh}. In particular, we have an isomorphism of $R$-modules
\[
\Hom^*_R(\End^*_{\cat T}(C),I(\fp)) \cong \Hom^*_{\cat T}(C,T_\fp C)\,.
\]
Since $\End^*_{\cat T}(C)$ is noetherian, $\Hom^*_{\cat T}(C,T_{\fp}C)$ is artinian. Moreover one of them has finite length if and only if the other does, by Matlis duality. This means that any compact object $C$ in $\cat T_\fp$ belongs to $\gam_\fp\cat T$. Thus $\gam_\fp\cat T=\cat T_\fp$, and therefore $\fp$ is minimal by our first observation.
\end{proof} 

\subsection*{Pure-injectivity}
Artinian objects enjoy  further finiteness properties. To explain this we recall the notion of purity for compactly generated triangulated categories; see \cite[\S1]{Krause:2000a}. A triangle $X\to Y\to Z\to$ is \emph{pure-exact} if for each compact object $C$ the induced sequence
\[
0\to\Hom_\cat T(C,X)\to\Hom_\cat T(C,Y)\to\Hom_\cat T(C,Z)\to 0
\] 
is exact. In that case $X\to Y$ is a \emph{pure monomorphism}. One calls an object $X$  \emph{pure-injective} if each pure monomorphism $X\to Y$ splits, and $X$ is \emph{$\Si$-pure-injective} if every coproduct of copies of $X$ is pure-injective.

\begin{proposition}
\label{pr:pure-inj}
Let $\cat T$ be a compactly generated $R$-linear triangulated category and $X\in \cat 
T$ artinian over $R$. Then $X$ is $\Si$-pure-injective and decomposes into a coproduct of indecomposable objects with local endomorphism rings.
\end{proposition}

\begin{proof}
Let $C\in\cat T$ be compact. A subgroup of $\Hom_\cat T(C,X)$ is said to be \emph{of finite definition} if it is the image of a map $\Hom_\cat T(C',X)\to \Hom_\cat T(C,X)$ that is induced by a morphism $C\to C'$ between compact objects. 
In \cite[Theorem~4.10]{Krause/Reichenbach:2001a} it is proved that when the descending chain condition on subgroups of finite definition holds for all compact $C$ in $\cat T$, the object $X$ is $\Si$-pure-injective and admits a decomposition into indecomposable objects with local endomorphism rings.  From this the assertion about artinian objects follows, using Remark~\ref{re:degree-zero}.
   \end{proof}

\begin{corollary}
Suppose that $\cat T$ admits a compact generator. Then the category of objects in $\cat T$ that are artinian over $R$ is a Krull--Schmidt category.
\end{corollary}
\begin{proof}
 Fix a compact generator $C$. We have already seen that any artinian object has a decomposition $X=\coprod_{i\in I} X_i$ into indecomposables. We need to show that $I$ is finite, but this is clear from the isomorphism $\Hom^*_\cat T(C,X)\cong \coprod_{i\in I}\Hom^*_\cat T(C,X_i)$ since this module is artinian.
\end{proof}

\section{Brown--Comenetz duality}
\label{se:bc-duality}
In this section we explain how classical Matlis duality from commutative algebra, between artinian modules and noetherian modules over complete local rings lifts to the world of tensor triangulated categories, leading to Brown--Comenetz duality. 
The motivation for this is our interest in the category of finite length objects in $\gam_\fp\cat T$. Three other categories arise naturally in their study:  artinian objects in $\gam_\fp\cat T$, noetherian objects in $\lam^\fp\cat T$, and a certain category of Matlis lifts of compacts objects, introduced below. These categories are related via Brown--Comenetz duality; see for instance Corollary~\ref{co:bcd+dg}.

To begin with we record basic results on Matlis duality. Since this is a local construction,  throughout this section we fix a graded commutative  noetherian ring $R$ and $\fp$ in $\Spec R$. We write $I(\fp)$ for the injective hull of $R/\fp$ and, for any $R$-module $M$, write $M\phat$ for the completion of $M_\fp$ with respect to $\fp R_\fp$. We remind the reader that   these constructions take place in the world of graded rings and modules.

\subsection*{Matlis duality}
On the category of $R$-modules \emph{Matlis duality at $\fp$} is the functor 
\[
M\longmapsto \bcd M\coloneqq\Hom_R^*(M,I(\fp))\,.
\] 
There is a canonical biduality map $M\to\bcd^2 M$; it is a monomorphism and we say $M$ is \emph{Matlis reflexive} when it is an isomorphism. For $M=R$ the map induces an isomorphism $R\phat\iso \End_R^*(I(\fp))$; this is part of Matlis duality which is given below. Therefore the Matlis dual $\bcd M$ admits a canonical $R\phat$-action. 

For any pair of $R$-modules $M,N$ there are natural isomorphisms of $R\phat$-modules 
\[
\Hom^*_R(M,\bcd N)\cong\Hom^*_R(M\otimes_R N,I(\fp))\cong\Hom^*_R(N,\bcd M)\,.
\]

Let $\art R$ and $\noeth R$ denote the categories of $R$-modules that are artinian and noetherian, respectively.
The result below is due to Matlis~\cite[\S4]{Matlis:1958a}.

\begin{proposition}
\label{pr:matlis}
    Let $R$ be a   graded commutative noetherian ring and $\fp\in\Spec R$. 
    Then the following  functors are mutually quasi-inverse contravariant equivalences:
\[
  \begin{tikzcd}[column sep = huge]
  \art R_\fp=\art R\phat
  \arrow[yshift=.75ex]{rrr}{\Hom^*_{R\phat}(-,I(\fp))=\Hom^*_{R}(-,I(\fp))}
  &&&\noeth R\phat
  \arrow[yshift=-.75ex]{lll}{\Hom^*_{R\phat}(-,I(\fp))}
  \end{tikzcd}
\]
Moreover, for a finitely generated $R$-module $M$ the canonical map $M\to \bcd^2 M$ induces an isomorphism $M\phat\iso \bcd^2 M$.\qedhere 
\end{proposition}

These constructions and results carry over to the world of triangulated categories. 

\subsection*{Matlis lifts}
Let $\cat T$ be a compactly generated $R$-linear triangulated category; it need not be tensor triangulated.  For each object $C$ in $\comp{\cat T}$ Brown representability yields an object $T_{\fp}C$ in ${\cat T}$ such that
\begin{equation}
\label{eq:inj-coh} 
\Hom_R(\Hom^*_{\cat T}(C,-), I(\fp))\cong\Hom_{\cat T}(-,T_{\fp}C)\,.
\end{equation} 
One can think of $T_{\fp} C$ as a \emph{Matlis lift} of $I(\fp)$ to $\cat T$, with respect to $C$. The assignment $C\mapsto T_{\fp}C$ yields a functor $\comp{\cat T}\to {\cat T}$, and it follows from Lemma~\ref{le:local-noetherian} that $T_\fp$ admits a canonical factorisation 
\[
\comp{\cat T}\twoheadrightarrow\comp{\cat T}_\fp\to\cat T_\fp\rightarrowtail\cat T\,.
\]
Matlis duality implies that when $\cat T$ is noetherian, $T_\fp C$ is an artinian object in the $R_\fp$-linear triangulated category $\cat T_\fp$ and belongs therefore to $\gam_\fp\cat T$; see also Lemma~\ref{le:t-artinian}. For tensor triangulated categories, one can lift Matlis duality to the level of objects, and not just on their cohomology. This is explained below.

\subsection*{Local Brown--Comenetz duality}
Let $\cat T$ be a rigidly compactly generated tensor triangulated category. We fix a   graded commutative noetherian ring $R$ and a canonical $R$-action on $\cat T$.

For an object $X\in\cat T$ Brown representability yields an object $\bcd{X}$ in ${\cat T}$ such that
\[
\Hom_R(\Hom^*_{\cat T}(\one,X\otimes -),I(\fp))\cong\Hom_{\cat T}(-,\bcd X).
\]
This yields a contravariant functor  $\bcd \colon{\cat T}\to{\cat T}$ called \emph{local Brown--Comenetz duality}. The same notation $\bcd$ is used for modules and for objects in a triangulated category, but there is no danger of confusion.

In the following we collect some basic properties which will be used throughout, sometimes without any further reference.

\begin{lemma}
\label{le:BC}
For $C\in \comp{\cat T}$ and $X\in \cat T$ there are natural isomorphism of $R$-modules
\[
  \Hom_{\cat T}^*(C,\bcd X) \cong \bcd\Hom^*_{\cat T}(\swd C,X)\,.
\]
\end{lemma}
\begin{proof}
The isomorphism is immediate from the defining isomorphism of $\bcd X$, using the dualisablity of $C$.
\end{proof}

For any pair of objects $X,Y$ in $\cat T$ there is a natural isomorphism of $R$-modules
\begin{equation*}
\label{eq:bcd}
\Hom^*_\cat T(X,\bcd Y)\cong \Hom^*_\cat T(Y,\bcd X)    
\end{equation*}
which means that $\bcd$ yields an adjoint pair of functors $\cat T^{\op} \rightleftarrows\cat T$. Also, it is easy to verify that there is a natural isomorphism
\[
\bcd X \cong\fHom(X,\bcd\one)\,.
\]
Therefore $\bcd$ restricts to a functor $\gam_\fp\cat T\to\lam^\fp\cat T$, thanks to \eqref{eq:loc-hom}. For any compact (equivalently, dualisable) object $C$, one has that
\[
\bcd C \cong  \swd C\otimes \bcd\one\,.
\]

The following observation connects Matlis lifts with Brown--Comenetz duality.

\begin{lemma}
\label{le:Tp-tensor}
 For any compact object $C$ in $\cat T$ there are natural  isomorphisms
\[
T_\fp C\cong C\otimes T_\fp\one\cong \bcd(\swd C)\,.
\]
\end{lemma}

\begin{proof}
Since $T_\fp\one \cong \bcd \one$, from definitions, one gets the isomorphism on the right. The one on the left holds because for each  $X\in\cat T$ there are natural isomorphisms:
\begin{align*}
\Hom_\cat T(X,T_\fp C)&\cong \Hom_R(\Hom^*_\cat T(C,X),I(\fp))\\   
&\cong \Hom_R(\Hom^*_\cat T(\one,\swd C\otimes X),I(\fp))\\ 
&\cong \Hom_\cat T(\swd C\otimes X,T_\fp \one)\\
&\cong \Hom_\cat T(X,C\otimes T_\fp\one)\,.
\end{align*}
The second and last isomorphisms hold because $C$ is dualisable.
\end{proof}

\subsection*{Completed linear actions}

We wish to establish a Matlis duality theorem  for $\cat T$ and need to assume that $\cat T$ is noetherian over $R$. Recall that this means compact objects are noetherian: the $R$-module $\Hom^*_{\cat T}(C,D)$ is finitely generated for all $C,D\in \comp{\cat T}$.  The first step is to show that the $\fp$-torsion category and the $\fp$-complete category admit $R\phat$-linear actions that extend the $R$-linear actions.

\begin{lemma}\label{le:DD-compact}
Let $\cat T$ be noetherian.
For a compact object $X\in\cat T_\fp$ the canonical morphism $X\to\bcd^2 X$ induces an isomorphism $\lam^\fp X\iso\bcd^2 X$.
\end{lemma}

\begin{proof}
We may assume that $R$ is local with maximal ideal $\fp$. Let $C\in\comp{\cat T}$. Since $\cat T$ is noetherian over $R$, the $R$-module $\Hom^*_\cat T(\swd C,X)$ is finitely generated. Hence, the $R$-module  $\Hom^*_\cat T(C,\bcd X)\cong\bcd\Hom^*_\cat T(\swd C,X)$ is artinian. In particular $\bcd X$ is $\fp$-torsion, and hence $\bcd^2 X\cong \fHom(\bcd X,\bcd\one)$ is $\fp$-complete, by \eqref{eq:loc-hom}.

It follows that the canonical morphism $X\to\bcd^2 X$ factors through $X\to\lam^\fp X$.
For $\lam^\fp X\to\bcd^2 X$ to be an isomorphism, it suffices to show that the induced map
\[
\Hom^*_{\cat T}(\kos C{\fp}, \lam^\fp X)
    \longrightarrow \Hom^*_{\cat T}(\kos C{\fp},\bcd^2 X)
\]
is a bijection; see Lemma~\ref{le:p-torsion}. For the same reason the induced map
\[
\Hom^*_{\cat T}(\kos C{\fp}, X)
    \longrightarrow \Hom^*_{\cat T}(\kos C{\fp},\lam^\fp X)
\]
is a bijection. But $\Hom^*_{\cat T}(\kos C{\fp}, X)$ is of finite length over $R$, so it is Matlis reflexive.
This yields the first isomorphism below
\[
\Hom^*_{\cat T}(\kos C{\fp}, X)\iso\bcd^2 \Hom^*_{\cat T}(\kos C{\fp}, X)\iso\Hom^*_{\cat T}(\kos C{\fp},\bcd^2 X)
\]
and the second is from Lemma~\ref{le:BC}.
\end{proof}

\begin{lemma}
The tensor triangulated categories $\gam_\fp\cat T$ and $\lam^\fp\cat T$ have canonical $R\phat$-linear actions, extending the $R$-action inherited from $\cat T$.
\end{lemma}

\begin{proof}
For each $X\in\cat T$  there is a sequence of homomorphisms 
\[
R\to R_\fp\to R\phat    \iso\End_R^*(I(\fp))\to\End^*_{\cat T}(\bcd\one)\to\End^*_\cat T(\bcd X)\,,
  \]  
where the isomorphism is from Matlis duality. This yields a factorisation of the canonical homomorphism $R\to\End^*_\cat T(\bcd X)$.
We have  $\lam^\fp\one_\fp \cong\bcd^2\one_\fp$ by Lemma~\ref{le:DD-compact}, and this gives the factorisation of $R\to\End^*_\cat T(\lam^\fp\one_\fp)$
by taking $X=\bcd\one_\fp$. This establishes the canonical $R\phat$-linear action for $\lam^\fp\cat T$, since $\lam^\fp\one_\fp$ is the tensor identity. The  equivalence $\lam^\fp\cat T\iso\gam_\fp\cat T$ from \eqref{eq:DG} preserves the tensor identity, since $\gam_\fp(\lam^\fp\one_\fp)\cong\gam_\fp\one$, and this yields the assertion for $\gam_\fp\cat T$.
\end{proof}

\subsection*{Duality over the completion}

In the following we restrict to $R$-modules and objects in $\cat T$ that admit an $R\phat$-action. Then we consider
the following modification of the Matlis duality:
\[
M\longmapsto D\phat M\coloneqq\Hom_{R\phat}^*(M,I(\fp))
\] 
Note that there is a natural morphism $D\phat M\to\bcd M$. This is an isomorphism when $M$ is $\fp$-torsion, because the class  of $R\phat$-modules $M$ such that this is an isomorphism includes the finite length modules and is closed under colimits, so it includes all $\fp$-torsion modules. There is also a canonical monomorphism $M\to (D\phat)^2 M$, which is an isomorphism when $M$ is artinian or noetherian. 

Analogously, for any object $X\in\cat T$ such that the canonical $R$-action on $\End^*_\cat T(X)$ factors through $R \to R\phat$, Brown representability yields an object $D\phat X$ in ${\cat T}$ with a natural isomorphism
\[
\Hom_{R\phat}(\Hom^*_{\cat T}(\one,X\otimes -),I(\fp))\cong\Hom_{\cat T}(-,D\phat X).
\]
The asssignment $X\mapsto D\phat X$ yields a pair of contravariant functors $\gam_\fp\cat T\to\cat T$ and $\lam^\fp\cat T\to\cat T$, equipped with natural transformations $D\phat \to\bcd$. 

\begin{lemma}
For  $X\in\gam_\fp\cat T$ the morphism $D\phat X\to\bcd X$ is an isomorphism. 
\end{lemma}
\begin{proof}
It suffices to show that the induced map $\Hom_{\cat T}^*(C,D\phat X)\to\Hom_{\cat T}^*(C,\bcd X)$ is an isomorphism for all compact $C\in\cat T$. But this follows from Lemma~\ref{le:BC} and its analogue for $D\phat$, since $\Hom_{\cat T}^*(C,X)$ is $\fp$-torsion.
\end{proof}

We formulate an analogue of Matlis duality for modules. To this end, consider the following subcategories of $\cat T$:
\begin{equation*}
\label{eq:AandN}
\begin{aligned}
 \cat A_\fp &\coloneqq \{X\in \gam_\fp\cat T \mid \Hom^*_{\cat T}(C,X)\in\art R\phat \text{ for all } C\in\comp{\cat T}\} \\
\cat N_\fp & \coloneqq \{X\in \lam^\fp\cat T\mid \Hom^*_{\cat T}(C,X)\in\noeth R\phat \text{ for all } C\in\comp{\cat T}\}
\end{aligned}
\end{equation*}

\begin{theorem}
\label{th:bc-duality}
Let $\cat T$ be a rigidly compactly generated tensor triangulated category with a canonical  $R$-action such that $\cat T$ is noetherian over $R$. Fix $\fp$ in $\Spec R$. Then there are mutually quasi-inverse contravariant equivalences: 
\[
  \begin{tikzcd}[column sep = huge]
  \cat A_\fp
  \arrow[yshift=.75ex]{r}{D\phat=D_\fp}
  &\cat N_\fp
  \arrow[yshift=-.75ex]{l}{D\phat}
  \end{tikzcd}
\]
Moreover
\[
\comp{(\gam_\fp\cat T)}=\comp{(\lam^\fp\cat T)}\subseteq \cat A_\fp\cap\cat N_\fp\,,
\]
and there are equalities
\[
(D\phat)^{-1}(\cat A_\fp)\cap\cat T_\fp=\cat N_\fp\qquad\text{and}\qquad (D\phat)^{-1}(\cat N_\fp)\cap\cat T_\fp=\cat A_\fp\,.
\]
\end{theorem}

\begin{proof}
We apply the analogue of Lemma~\ref{le:BC} for $D\phat$ and use the Matlis duality from Proposition~\ref{pr:matlis}. In particular, for objects with an $R\phat$-action, the assignment $X\mapsto D\phat X$ sends artinian to noetherian objects, and vice versa.
Any artinian $R_\fp$-module is $\fp$-torsion, so $D\phat X\in\cat A_\fp$ for $X\in\cat N_\fp$.
On the other hand, for  $X\in\gam_\fp\cat T$ there are isomorphisms
\begin{equation*}
\label{eq:double}
D\phat X\cong \bcd X\cong \fHom(X,\bcd\one)\,,
\end{equation*}
and this lies in $\lam^\fp\cat T$. Thus $D\phat X\in\cat N_\fp$ for $X\in\cat A_\fp$.
For a compact $C\in\cat T$ there is a natural isomorphism
\[
 \Hom_{\cat T}^*(C,(D\phat)^2 X) \cong (D\phat)^2\Hom^*_{\cat T}(C,X)\,.
\]
Thus the canonical morphism $X\to (D\phat)^2 X$ is an isomorphism when $X$ is in $\cat A_\fp$ or in $\cat N_\fp$. This gives the desired pair of equivalences.

For the statement about the categories of compacts, we recall from Lemma~\ref{le:p-torsion} that $\comp{(\gam_\fp\cat T)}=\comp{(\lam^\fp\cat T)}$, and  $\comp{(\gam_\fp\cat T)}\subseteq \cat A_\fp\cap\cat N_\fp$ holds since $\cat T$ is noetherian. For the final pair of equalities, let $X\in\cat T_\fp$ such that $D\phat X$ is in $\cat A_\fp$. Then $(D\phat)^2 X$ is in $\cat N_\fp$. The map $\Hom^*_{\cat T}(C,X)\to \Hom^*_{\cat T}(C,(D\phat)^2 X)$ is a monomorphism for all compact $C$, again by Lemma~\ref{le:BC}. Thus $X$ is in $\cat N_\fp$. The other inclusion $(D\phat)^{-1}(\cat A_\fp)\cap\cat T_\fp\supseteq\cat N_\fp$ is clear. The proof of the equality for $\cat A_\fp$ is analogous.
\end{proof}

The following consequence is pertinent to the discussion in the next section.

\begin{corollary}
\label{co:bcd+dg}
For $\fp$ in $\Spec R$, there are equivalences of triangulated categories
 \[
    \begin{tikzcd}
    \thick(T_\fp{(\comp{\cat T})})^\op = \thick(\bcd{(\comp{\cat T}_\fp)})^\op \arrow[rightarrow,yshift=-.75ex,swap,r,"D\phat"]
	   & \thick(\lam^\fp{(\comp{\cat T}_\fp)}) \arrow[rightarrow,yshift=.75ex,l,swap,"D\phat"]          \arrow[rightarrow,yshift=-.75ex,swap,r,"\gam_\fp"]
	&  \thick(\gam_\fp{(\comp{\cat T}_\fp)})\,. \arrow[rightarrow,yshift=.75ex,l,swap,"\lam^\fp"]
    \end{tikzcd}
    \]
\end{corollary}

\begin{proof}
The equality on the left follows from Lemma~\ref{le:Tp-tensor}. The next equivalence uses Lemma~\ref{le:DD-compact}
and is a restriction of the equivalence from Theorem~\ref{th:bc-duality}. The equivalence on the right is obtained from the Dwyer--Greenlees correspondence~\eqref{eq:DG}.
\end{proof}

\subsection*{Noetherian objects} 

The result below, which is an analogue of Proposition~\ref{pr:artinian}, characterises objects in the local category $\lam^\fp \cat T$ with noetherian cohomology.

\begin{proposition}
\label{pr:noetherian}
Fix  $\fp\in\Spec R$. An object $X\in\lam^\fp\cat T$ is in $\cat N_\fp$ if and only if $X$ has finite length over $R_\fp$.
\end{proposition}

\begin{proof}
The module $\Hom^*_{\cat T}(C,X)$ is $\fp$-torsion by Lemma~\ref{le:p-torsion}, and therefore the forward implication  is clear.

Suppose the $R_\fp$-module $\Hom^*_{\cat T}(C,X)$ has finite length for each compact object $C$ in $\lam^\fp{\cat T}$. Then Lemma~\ref{le:BC} implies that $\Hom^*_{\cat T}(C,\bcd X)$ has finite length over $R_\fp$ for each compact object $C$ in $\gam_\fp\cat T$. Thus Proposition~\ref{pr:artinian} yields that the object $\gam_\fp \bcd X$ is artinian over $R_\fp$. Then its Brown--Comenetz dual is in $\cat N_\fp$, by Theorem~\ref{th:bc-duality}.

The natural map $\gam_\fp \bcd X\to \bcd X$ induces the map on the right below:
\[
X\longrightarrow \bcd^2 X \longrightarrow \bcd{\gam_\fp \bcd X}\,.
\]
The one of the left is the biduality map.  So it suffices to verify that the composite $X\to \bcd{\gam_\fp\bcd X}$ is an isomorphism.
Since source and target are in $\lam^\fp\cat T$, it suffices to verify that the induced map
\[
\Hom_{\cat T}^*(C,X)\longrightarrow \Hom_{\cat T}^*(C, \bcd{\gam_\fp\bcd X})
\]
is an isomorphism for each compact object $C$ in $\lam^\fp\cat T$. It remains to observe that one has isomorphisms
\begin{align*}
\Hom_{\cat T}^*( C, \bcd{\gam_\fp\bcd X})
    &\cong \bcd \Hom_{\cat T}^*(\swd C, \gam_\fp\bcd X) \\
    &\cong \bcd \Hom_{\cat T}^*(\swd C, \bcd X) \\
    &\cong \bcd^2 \Hom_{\cat T}^*(C, X)\\
    &\cong \Hom_{\cat T}^*(C, X)\,.
\end{align*}
The first and the third isomorphism follow from  Lemma~\ref{le:BC}. The second isomorphism uses that $\swd C$ is in $\gam_\fp\cat T$, and the last one holds because the $R_\fp$-module $\Hom_{\cat T}^*(C,X)$ has finite length. 
\end{proof}

\subsection*{Dualisable objects}
Our goal is a cohomological description of dualisable objects in a local stratum of a tensor triangulated category, and here is a first step. 

\begin{proposition}
      \label{pr:dualising-hom}
 Let $\cat T$ be a rigidly compactly generated tensor triangulated category with a canonical $R$-action such that $\cat T$ is noetherian over $R$. Fix $\fp$ in $\Spec R$. The following statements hold.
 \begin{enumerate}[\quad\rm(1)]
     \item 
     The dualisable objects of $\gam_\fp{\cat T}$ are in $\cat A_\fp$, and the dualisable objects of $\lam^\fp{\cat T}$ are in $\cat N_\fp$. 
     \item 
     There are inclusions $\gam_\fp(\comp{\cat T}_\fp)\subseteq \cat A_\fp$  and $\lam^\fp(\comp{\cat T}_\fp)\subseteq \cat N_\fp$.
 \end{enumerate} 
 \end{proposition}

\begin{proof}
(1) Let $\cat S$ denote either $\gam_\fp\cat T$ or $\lam^\fp\cat T$ and view it as an $R_\fp$-linear category. Fix a dualisable object $X\in\cat S$. Given Propositions~\ref{pr:artinian} and \ref{pr:noetherian} it suffices to verify that the $R_\fp$-module  $\Hom^*_{\cat S}(C,X)$ has finite length for each compact object $C$. Consider the isomorphism of $R_\fp$-modules
\[
\Hom_{\cat S}^*(C,X) \cong \Hom^*_{\cat S}(\one,\swd C \otimes  X)\,,
\]
where the unit and tensor product are the ones in $\cat S$. The object $Y\coloneqq\swd C\otimes X$ is compact in $\cat S$, by Proposition~\ref{pr:dualisable}. Lemma~\ref{le:p-torsion} gives
\[
\comp{(\gam_\fp\cat T)}=\comp{\cat S}=\comp{(\lam^\fp\cat T)}\,.
\]  
This yields an isomorphism $\Hom_{\cat S}^*(\one,Y)\cong\Hom_{\cat T_\fp}^*(\one_\fp,Y)$ of $R_\fp$-modules, where $\one_\fp$ denote the unit in $\cat T_\fp$, and $\Hom_{\cat T_\fp}^*(\one_\fp,Y)$ has finite length over $R_\fp$.

(2) We apply part (1). Any object in $\gam_\fp(\comp{{\cat T}}_\fp)$ is dualisable and therefore in $\cat A_\fp$. In the same vein, any object in $\lam^\fp(\comp{{\cat T}_\fp})$ is dualisable and therefore in $\cat N_\fp$.
\end{proof}

The diagram below illustrates the findings in this section. Corollary~\ref{co:compact-dualisable} provides a context where the bottom-most inclusion is an equality.
\begin{figure}[ht]
\centering
\[
\begin{tikzcd}[column sep = 3ex, row sep = 3ex]
&\cat T_\fp\\
&\comp{\cat T}_\fp\arrow[hook]{u}\arrow{ldd}{\gam_\fp}\arrow[swap]{rdd}{\lam^\fp}\\
\gam_\fp\cat T\arrow[hook]{ruu}&&\lam^\fp\cat T\arrow[hook]{luu}\\
\cat A_\fp\arrow[hook]{u}\arrow[leftrightarrow,swap]{rr}{D\phat}&&\cat N_\fp\arrow[hook]{u}\\
{}\\
&\cat A_\fp\cap\cat N_\fp\arrow[hook]{luu}\arrow[hook]{ruu}\\
&\comp{(\gam_\fp\cat T)}=\comp{(\lam^\fp\cat T)}\arrow[hook]{u}
\end{tikzcd}
\]
\caption{Subcategories of $\cat T_\fp$}\label{fi:figure}
\end{figure}

\section{Strong generation}
\label{se:strong-gen}

We introduce a notion of local regularity for compactly generated triangulated categories which builds on the notion of strong generation. A key consequence of this property is a strong finiteness condition for the class of artinian objects, analogous to finite global dimension in the setting of abelian categories. We begin by recalling the notion of a strong generator from \cite{Bondal/vandenBergh:2003a,Rouquier:2008a}.

\subsection*{Strong generation}
Fix a triangulated category $\cat T$ with suspension $\Si\colon\cat T
\iso\cat T$. For an object $X$ in $\cat T$ and $n\ge 0$ set
\[
\thick^n(X)\coloneqq\begin{cases}
  \{0\} &n=0,\\
  \add\{\Si^i X\mid i\in\bbZ\}&n=1,\\
  \add\{\cone\phi\mid\phi\in\Hom_\cat T(\thick^1(X), \thick^{n-1}(X))&n>1,
\end{cases}\]
where $\add\cat X$ denotes the smallest full subcategory containing $\cat X$ that is closed under finite direct sums and direct summands. An object $X$ is a \emph{strong generator} for $\cat T$ if $\cat T=\thick^n(X)$ for some integer $n\ge 0$. The category $\cat T$ is \emph{strongly generated} if it admits a strong generator; in this case any object $X$ satisfying  $\thick(X)=\cat T$ is a strong generator.

\begin{example}\label{ex:gldim}
    Let $\cat A$ be an abelian category with an injective cogenerator $I$ such that every object embeds into a finite direct sum of copies of $I$. Viewing $I$ as a complex concentrated in degree zero, the equality $\bfD^{\mathrm b}(\cat A)=\thick(I)$ holds if and only if each object in $\cat A$ has finite injective dimension. Moreover, $\thick(I)$ is strongly generated if and only if $\cat A$ has finite global dimension. This is a consequence of the ghost lemma, cf.\ \cite[Lemma 2.4]{Krause/Kussin:2006a}. 
    
    There is also a dual statement when $\cat A$ has a projective generator. It implies, for example, that the category of perfect complexes over a local noetherian ring admits a strong generator if and only if the ring is regular; see  \cite[Proposition~7.25]{Rouquier:2008a}. 
\end{example}

\subsection*{Artinian objects}

Let $\cat T$ be a compactly generated $R$-linear triangulated category. The subcategory of objects that are artinian over $R$ is denoted $\art_R\cat T$. 

\begin{lemma}
\label{le:t-artinian}
When the $R$-linear category $\cat T$ is noetherian one has inclusions
\begin{gather*}
    \thick\{\gam_\fm C\mid C\in\comp{\cat T},\,\fm\in \Spec R 
            \text{ maximal}\}\subseteq \art_R\cat T \\
    \thick\{T_\fm C\mid C\in\comp{\cat T},\,\fm\in \Spec R 
            \text{ maximal}\}\subseteq \art_R\cat T\,.
\end{gather*}
\end{lemma}

\begin{proof}
Since $\art_R \cat T$ is a thick subcategory of $\cat T$, it suffices to verify that it contains $\gam_\fm C$ for $\fm$ maximal and $C$ compact. 
Using Proposition~\ref{pr:artinian} it suffices to show that the $R_\fm$-module $\Hom^*_{\cat T_\fm}(X,\gam_\fm C)$ has finite length for all $X\in\comp{(\gam_\fm\cat T)}$. We have $\Hom^*_{\cat T_\fm}(X,\gam_\fm C)\cong\Hom^*_{\cat T_\fm}(X,C_\fm)$ for such $X$. This module is noetherian and $\fm$-torsion, so of finite length.
For the other inclusion, one uses the isomorphism defining $T_\fm C$. Matlis duality yields that $\Hom_R^*(-,I(\fm))$ takes finitely generated modules to artinian modules, since $I(\fm)$ is artinian.
\end{proof}

\begin{definition}
\label{def:Rlocally}
 We say that $\cat T$ is \emph{locally regular over $R$} if it is noetherian and has a compact generator $G$ such that $\thick(T_\fp G)$ is strongly generated for each $\fp\in\Spec R$.
\end{definition}

Local regularity has useful consequences for the class of artinian objects. In particular, $\thick(T_\fp G)$ does not depend on the choice of the compact generator $G$.

\begin{proposition}
\label{pr:strong-regular}
    Let $R$ be a graded local commutative noetherian ring  with maximal ideal $\fm$, and $\cat T$ a compactly generated $R$-linear triangulated category.  Assume $\cat T$ is noetherian and has a compact generator $G$. If $\thick(T_\fm G)$ is strongly generated, then  $\thick(T_\fm G)=\art_R\cat T$. If in addition $\thick(\gam_\fm G)$ is strongly generated, then  \[\thick(\gam_\fm G)=\art_R\cat T=\thick(T_\fm G)\,.\]
\end{proposition}

\begin{proof}
We write $R\mhat$ for the completion of $R$ at $\fm$ and note that $\End^*_R(I(\fm))\cong R\mhat$. Set  $\cat I=\thick(T_\fm G)$ and fix $X\in\art_R\cat T$.
Then $\Hom^*_\cat T(X,Y)$ is naturally a finitely generated graded $R\mhat$-module for all $Y\in\cat I$. This yields a cohomological functor
\[
\Hom^*_{\cat T}(X,-)|_{\cat I}\colon\cat I\lra\noeth R\mhat\,.
\] 
 Since $\cat I$ is strongly generated, a variation of Brown representability due to Rouquier \cite{Rouquier:2008a} and extended by Letz to the graded setting in \cite[Theorem~2.7]{Letz:2023a}, implies that this functor is graded representable.  Let $X'$ denote the representing object. By Yoneda's lemma, an isomorphism
  \[
  \Hom^*_{\cat I}(X',-)\iso \Hom^*_{\cat T}(X,-)|_{\cat I}
  \] 
induces a morphism $X\to X'$, which is an isomorphism since $T_\fm G$ is a cogenerator for $\gam_\fm\cat T$.
Thus $X$ belongs to $\cat I$. 

The proof for $\cat J=\thick(\gam_\fm G)$ is analogous, using cohomological functors $\cat J^\op\to \noeth R\mhat$. Note that  $\Hom^*_\cat T(X,Y)$ is naturally a finitely generated graded $R\mhat$-module  for all $X,Y\in\cat J$ since  $\cat J\subseteq\cat I$.
\end{proof}

\subsection*{Hom-symmetry} 

The following symmetry statement is needed later. It assumes a weak Gorenstein property that holds when $\cat T$ is a locally regular tensor triangulated category; see Lemma~\ref{le:reg-bcd+dg} and Proposition~\ref{pr:strong-regular}

\begin{lemma}\label{le:hom-symmetry}
Let $\fp\in\Spec R$ and assume
\[
\thick\{\gam_\fp C\mid C\in\comp{\cat T}\}=\thick\{T_\fp C\mid C\in\comp{\cat T}\}\,.
\]
Then the following holds for each $X\in\gam_\fp\cat T$.
\begin{enumerate}[\quad\rm(1)]
\item The $R_\fp$-module $\Hom^*_{\cat T}(C,X)$ has finite length for each $C\in\comp{\cat T}$ if and only if $\Hom^*_{\cat T}(X,C)$ has finite length for each $C\in\comp{\cat T}$. 
\item The $R_\fp$-module $\Hom^*_{\cat T}(C,X)$ has finite length for each $C\in\comp{(\gam_\fp\cat T)}$ if and only if $\Hom^*_{\cat T}(X,C)$ has finite length for each $C\in\comp{(\gam_\fp\cat T)}$. 
\end{enumerate}

\end{lemma}
\begin{proof}
(1) The $R_\fp$-module $\Hom^*_{\cat T}(C,X)$ has finite length if and only if  
\[
\bcd\Hom^*_{\cat T}(C,X)\cong \Hom^*_{\cat T}(X,T_\fp C)
\] 
has finite length. The assumption on $\cat T$ implies that it is equivalent to the finiteness of the length of
\[
\Hom^*_{\cat T}(X,\gam_\fp C)\cong\Hom^*_{\cat T}(X,C)\,.
\]

(2) One has $\comp{(\gam_\fp\cat T)}=\thick\{\kos{C_\fp}{\fp}\mid C\in\comp{\cat T}\}$, by  Lemma~\ref{le:p-torsion}. It then follows from Lemma~\ref{le:koszul-tt} that the $R_\fp$-module $\Hom^*_{\cat T}(C,X)$ has finite length for each $C\in\comp{(\gam_\fp\cat T)}$ if and only if 
  $\Hom^*_{\cat T}(C,\kos{X}{\fp})$ has finite length for each $C\in\comp{\cat T}$. 
  The same argument applies to $\Hom^*_{\cat T}(X,C)$. So we have reduced to the statement of part (1).
\end{proof}

\subsection*{Descent}
Let  $f \colon \cat T\to \cat U$ be an exact functor between $R$-linear triangulated categories. We say that $f$ is \emph{$R$-linear} if for each object $X$ in $\cat T$ the canonical map 
\[
\End^*_{\cat T}(X)\lra \End^*_{\cat U}(fX)
\]
is $R$-linear. An equivalent condition is that these are  maps of $R$-algebras.

In the following we study adjoint pairs $(f,g)$ of functors between compactly generated  triangulated categories. We  make frequent use of the fact that $f$ preserves compacts if and only if $g$ preserves coproducts.

\begin{proposition}
\label{pr:regularity}
Let $f\colon \cat T\rightleftarrows \cat U\cocolon g$ be an adjoint pair of $R$-linear exact functors
between compactly generated $R$-linear triangulated categories such that both functors preserve products and coproducts. Suppose there exists an integer $d$ such that for each $X\in\cat T$ one has $X\in \thick^d(gfX)$.  If $\cat U$ is locally regular over $R$, then so is $\cat T$.
\end{proposition}

A sufficient condition for $(f,g)$ to preserve products and coproducts is the existence of an equivalence $u\colon\cat U\iso\cat U$ such that $(g,uf)$ is an adjoint pair, because then the pair $(gu,f)$ is adjoint.

\begin{proof}
First observe that $f$ admits a left adjoint $e\colon\cat U\to\cat T$ that preserves compacts. This follows from Brown representability since $f$ preserves products and coproducts. Also, $f$ preserves compacts since $g$ preserves coproducts.

Assume $\cat U$ is noetherian over $R$. By adjunction, the $R$-module $\Hom^*_{\cat T}(C,gfD)$ is noetherian for each pair of compact objects $C,D$ in $\cat T$. Since $D$ is in $\thick(gfD)$ we conclude that that $\Hom^*_{\cat T}(C,D)$ is noetherian. Thus $\cat T$ is noetherian over $R$.

Suppose that $\cat U$ is locally regular with compact generator $G$. Fix $\fp\in\Spec R$. There exists an integer $n$ such that
\[
\thick^n(T_\fp G)=\thick(T_\fp G) = \art_{R_{\fp}}\cat U_\fp\,,
\]
where the first equality is from the definition of local regularity and the second one is by Propostion~\ref{pr:strong-regular}. We claim
\begin{enumerate}
    \item  
    $eG$ is a compact generator for $\cat T$,
    \item 
    $g(T_\fp G)\cong T_\fp(eG)$, and
    \item
    $\thick^{d+n}(T_\fp(eG))=\thick(T_\fp (eG))$,
\end{enumerate}  
which would prove that $\cat T$ is locally regular over $R$, as desired.

Indeed, (1) holds because $e$ preserves compactness and for each $X\in \cat T$ one has
an isomorphism of $R$-modules
\[
\Hom^*_{\cat U}(G,fX)  \cong \Hom_{\cat T}(eG,X)\,.
\]
This isomorphism also yields isomorphisms
 \begin{align*}
    \Hom^*_{\cat T}(X,g(T_\fp G)) 
        &\cong \Hom^*_{\cat U}(fX, T_\fp G) \\
        &\cong \Hom_R(\Hom^*_{\cat U}(G,fX), I(\fp)) \\
        &\cong \Hom_R(\Hom^*_{\cat T}(eG,X), I(\fp))\,.
\end{align*}
 Therefore $g(T_\fp G)\cong T_\fp(eG)$ holds, which justifies (2).

 (3) Given Lemma~\ref{le:t-artinian} it suffices to prove that each artinian object $X$ in $\cat T_\fp$ is in $\thick^{d+n}(T_\fp(eG))$. It is easy to verify that $fX$ is in 
   $\art_{R_\fp}\cat U_\fp$, and hence in $\thick^n(T_\fp G)$. Thus we get that
\[
 X\in \thick^d(gfX) \subseteq \thick^{d+n}(g(T_\fp G))=\thick^{d+n}(T_\fp(eG))\,,
\]
where the equality is by (2). This justifies (3).
\end{proof}

\subsection*{Change of rings}
Let  $\varphi\colon R\to S$ be a finite morphism of graded commutative noetherian rings. The following lemma shows that local regularity is preserved when the linear action on a triangulated category is changed via $\varphi$.

\begin{lemma}
\label{le:change-of-rings}
Let $\cat T$ be a compactly generated $S$-linear triangulated category. Then one has $\art_R\cat T=\art_S\cat T$, and for each $\fp\in \Spec R$ and compact object $C$ in $\cat T$ there is an isomorphism
\[
T_{\fp}C\cong \bigoplus_{\substack{\fq\in\Spec S\\ \varphi^{-1}(\fq)=\fp}} T_\fq C\,.
\]
Moreover $\cat T$ is  locally regular over $R$ if and only if the same holds over $S$.
\end{lemma}

\begin{proof}
The first assertion is clear, because an $S$-module is artinian if and only if it is artinian when viewed as an $R$-module via restriction of scalars along $\varphi$.
One has an isomorphism of $S$-modules 
\[
\Hom_R(S,I(\fp)) \cong \bigoplus_{\substack{\fq\in\Spec S\\ \varphi^{-1}(\fq)=\fp}} I(\fq)\,.
\]
See, for instance, \cite[Theorem~1.1]{Rahmati:2009a}. Combining it with the adjunction isomorphism
\[
\Hom_R(\Hom^*_\cat T(C,-),I(\fp))\cong\Hom_S(\Hom^*_\cat T(C,-),\Hom_R(S,I(\fp)))
\]
and the defining isomorphism of $T_\fp C$, yields the isomorphism concerning $T_\fp C$.

Let $\fp\in\Spec R$. We localise at $\fp$ and may assume $R$ is local with maximal ideal $\fp$. Consider the induced map $\Spec S\to \Spec R$. Since $R\to S$ is finite, the fibre over $\fp$ consists of a finite collection of maximal ideals, say $\fq_1,\dots,\fq_r$. It follows that any object $X$ in $\cat T$ admits a functorial decomposition
    \[
    \gam_\fp X = \gam_{\fq_1}X \oplus \cdots \oplus \gam_{\fq_r}X\,.
    \]
In particular, any exact triangle supported at $\fp$ decomposes into a direct sum of exact triangles supported at the $\fq_i$.    
When applied to $T_\fp C$ for a compact object $C$, one obtains the above decomposition $T_\fp C\cong T_{\fq_1}C \oplus \cdots \oplus T_{\fq_r}C$. Here, one uses that $T_\fp C$ is in $\gam_\fp\cat T$, so $\gam_\fp T_\fp=T_\fp$.

Now suppose $\cat T$ is locally regular over $S$, with compact generator $G$. This yields integers $n_i$ such that
\[
\thick^{n_i}(T_{\fq_i} G)=\thick(T_{\fq_i} G)\qquad (i=1,\ldots,r)\,.
\]
Using the above decomposition $\gam_\fp=\bigoplus_i\gam_{\fq_i}$, each object in $\thick(T_{\fp} G)$ decomposes into objects from the categories $\thick(T_{\fq_i} G)$.
Hence $\thick^{n}(T_{\fp} G)=\thick(T_{\fp} G)$ for $n=\max(n_1,\ldots,n_r)$, and $\cat T$ is regular over $R$. 

For the converse suppose that $\cat T$ is locally regular over $R$, with compact generator $G$, and fix $\fq\in\Spec S$. Set $\fp=\varphi^{-1}(\fq)$. As above, we can localise at $\fp$ and assume $R$ is local with maximal ideal $\fp$, and that $\fq$ lies over $\fp$. It is easy to verify that 
\[
\art_{S_\fq}\cat T_\fq\subseteq \art_S \cat T = \art_R\cat T\,.
\]
It follows from Proposition~\ref{pr:strong-regular} that each artinian object $X$ in the $S_\fq$-linear category $\cat T_\fq$ belongs to $\thick^n(T_\fp G)$ for some integer $n$. The decomposition of $T_\fp G$ noted above implies that $\gam_\fq T_\fp G=T_\fq G$, so applying the functor $\gam_{\fq}$ yields that $X$ belongs to $\thick^n(T_\fq G)$. Thus $\cat T$ is locally regular over $S$.
\end{proof}

\subsection*{Regular endomorphism rings}
Let $\cat T$ be a compactly generated triangulated category and $G$ a compact generator. Set $A=\End^*_\cat T(G)$ and consider the cohomological functor 
\[
H\colon\cat T\lra \Mod A,\quad\text{where}\quad X\mapsto \Hom^*_\cat T(G,X).
\]
For an injective $A$-module $E$ let $T(E)$ denote the object in $\cat T$ such that
\[
\Hom_A(H(X),E)\cong\Hom_\cat T(X,T(E)).
\]
The existence and uniqueness of $T(E)$ follows from Brown representability.

The result below uses Adams resolutions and is well known; see for instance \cite{Miller:1981a}.

\begin{lemma}
\label{le:adams}
Let $X\in\cat T$ be an object such that the $A$-module $H(X)$ admits an injective resolution $\eta\colon H(X)\to I$ where $I^s=0$ for $s> n$. Then $X$ belongs to the triangulated subcategory of $\cat T$ generated by $T(I^0),\ldots,T(I^n)$.
\end{lemma}

\begin{proof}
The argument is an induction on $n$. The morphism $\eta\colon H(X)\to I^0$ corresponds to a morphism $\widetilde \eta\colon X\to T(I^0)$.
If $n=0$, then $\eta$ is an isomorphism. It follows that $\Hom_{\cat T}^*(G,\widetilde\eta)$ is an isomorphism, and hence so is $\widetilde\eta$, because $G$ generates $\cat T$. 

Assume $n\ge 1$ and let $X'$ denote the cone of the morphism $\widetilde\eta$. This yields an exact sequence $0\to H(X)\to I^0\to H(X')\to 0$. The induction hypothesis implies that $X'$ belongs to the triangulated subcategory generated by $T(I^1),\ldots,T(I^n)$. Thus $X$ belongs to the triangulated subcategory generated by $T(I^0),\ldots,T(I^n)$.
\end{proof}

\begin{theorem}
\label{th:reg-end}  
Let $\cat T$ be a compactly generated $R$-linear category that is noetherian. If there exists a compact generator $G$ for $\cat T$ such that $\End^*_{\cat T}(G)$ has finite global dimension, then $\cat T$ is locally regular over $R$.   
\end{theorem} 

\begin{proof}   
Set $A=\End^*_{\cat T}(G)$ and fix $\fp \in \Spec R$. Then $G_\fp$ is a compact generator for $\cat T_\fp$, and one has an isomorphism of rings
\[
\End^*_{\cat T_\fp}(G_\fp)\cong \End^*_{\cat T}(G)_\fp =  A\otimes_R R_\fp\,.
\]
Since $A\otimes_R R_\fp$ is also of finite global dimension, we can replace $R$ and $\cat T$ by their localisations at $\fp$ and assume $R$ is local, with maximal ideal $\fp$. The $A$-module $J=\Hom_R(A,I(\fp))$ is injective, and it is straightforward to verify that $T_\fp G = T(J)$. The desired result is thus that $\thick(T(J))$ has a strong generator.  Since $\cat T$ is noetherian, arguing as in the proof of Lemma~\ref{le:t-artinian} one deduces that $T(J)$, and hence any object in $\thick(T(J))$, is artinian.

Let $X$ be any artinian object in $\cat T$ and set $M=\Hom^*_{\cat T}(G,X)$. Then $M$ is artinian as an $R$-module, so it embeds into a finite direct sum of copies of $I(\fp)$. Thus as an $A$-module $M$ embeds into a finite direct sum of copies of $J$. The cokernel of such an embedding is again artinian over $R$, and we obtain therefore an injective resolution of $M$ over $A$ involving in each degree finite direct sums of copies of $J$ and their direct summands. Thus, if the global dimension of $A$ is bounded by $n$, then $X$ belongs to $\thick^{n+1}(T(J))$, as desired; see Lemma~\ref{le:adams} and its proof. 
\end{proof}

\section{Local regularity}
\label{se:regularity}

In this section we consider a rigidly compactly generated tensor triangulated category $\cat T=(\cat T,\otimes,\one)$ and introduce in this setting a notion of local regularity. Then we are able to describe for each point its dualisable objects.

\begin{definition}
\label{def:locally}
Set $R=\End^*_{\cat T}(\one)$. We say $\cat T$ is \emph{noetherian} to mean that $\cat T$ is noetherian with respect to the canonical $R$-action, so the graded $R$-module $\Hom^*_{\cat T}(C,D)$ is noetherian for all compact objects $C,D$ in $\cat T$. In particular, the graded ring $R$, which is always commutative, is noetherian. 

We say that $\cat T$ is \emph{locally regular} if it is noetherian and admits a compact generator $G$ such that $\thick(\gam_\fp G)$ is strongly generated for each $\fp\in\Spec R$.
\end{definition}

For any pair $G,G'$ of compact generators, one has $\thick(\gam_\fp G)=\thick(\gam_\fp G')$ for each $\fp\in\Spec R$. Thus the definition of local regularity does not depend on the choice of a generator. 

The result below reconciles this notion of local regularity with the one introduced in Section~\ref{se:strong-gen} for triangulated categories with linear actions. For tensor triangulated categories, the definition above is more intrinsic. In fact, the linear action of $\End^*_{\cat T}(\one)$ plays a major role in our arguments. If this ring is too small, then the definition ought to be formulated in terms of prime ideals of the tensor triangulated category $\comp{\cat T}$. An example to keep in mind is the derived category of a non-affine scheme. Another such example from representation theory is the derived category of a Schur algebra, with subcategory of compact objects equivalent to the homotopy category of Young permutation modules. 

\begin{lemma}
\label{le:reg-bcd+dg}
Let $\cat T$ be a rigidly compactly generated tensor triangulated category. Assume $\cat T$ is noetherian and has a compact generator $G$. Then the local regularity of $\cat T$ is equivalent to strong generation of any one, equivalently, all of the triangulated categories:
 $\thick(T_\fp G)$,  $\thick(\bcd G)$ and $\thick(\lam^\fp G_\fp)$, for each $\fp\in\Spec R$.
\end{lemma}

\begin{proof}
This follows from Corollary~\ref{co:bcd+dg}.
\end{proof}

\subsection*{Examples}

The definition of local regularity is motivated by the following example. For a noetherian ring $A$ we write  $\bfD(\Mod A)$ for its full derived category; it is a rigidly compactly generated tensor triangulated category.

\begin{proposition}\label{pr:comm-regular}
Let $A$ a commutative noetherian ring. Then the ring $A$ is regular
if and only if the tensor triangulated category $\bfD(\Mod A)$ is locally regular.
 \end{proposition}

\begin{proof}
Local regularity is a local condition. For each $\fp\in\Spec A$ we have 
\[
\bfD(\Mod A)_\fp\iso\bfD(\Mod A_\fp)\,.
\]
Thus we may assume $A$ is local with maximal ideal $\fp$; set $k=A/\fp$.

The triangulated category $\bfD(\Mod A)$ has a compact generator, namely, $A$, and is noetherian as an $A$-linear category. 
It is convenient to check local regularity via the strong generation of $\thick(T_\fp A)$; see Lemma~\ref{le:reg-bcd+dg}.
The $A$-module $I(\fp)$ identifies with $T_\fp A$, and is an injective cogenerator of $\art A$, the category of artinian $A$-modules.  Moreover, an $A$-complex $X$ is artinian precisely when the $A$-module $H^*(X)$ is artinian, so $\art_A\bfD(\Mod A)$ identifies with $\bfD^{\mathrm b}(\art A)$. Thus, given Example~\ref{ex:gldim}, it remains to verify that $A$ is regular if and only if $\art A$ has finite global dimension. This follows from the result due to Auslander, Buchsbaum, and Serre that $A$ is regular if and only if $\Ext_A^n(k,k)=0$ for $n\gg 0$; see Proposition~1.3.1 and Corollary~2.2.7 in \cite{Bruns/Herzog:1998a}. 
\end{proof}

Here is another example of a locally regular category.

\begin{example} 
\label{ex:exterior} 
Let $k$ be a field and $A = \Lambda^*(V)$ the exterior algebra on a finite dimensional $k$-vector space $V$ which we view as a dg-algebra generated in degree $-1$ with trivial differential. Let $\cat T = \KInj{A}$ denote the homotopy category of injective $A$-modules. This category is compactly generated and the Hopf algebra structure  gives rise to a tensor product. The category $\cat T$ is noetherian and generated by the injective resolution of the trivial $A$-module $k$. The graded endomorphism ring identifies with a symmetric algebra, so $\End_\cat T^*(\ik) \cong S^*(V^\vee[2])$. Hence  $\cat T$ is locally regular by Theorem~\ref{th:reg-end}.
\end{example}

\subsection*{Descent}

The result below is a version of Proposition~\ref{pr:regularity} for tensor triangulated categories. In the statement, the condition that $f$ is strong symmetric monoidal means that $f(\one_{\cat T})\cong \one_{\cat U}$ and there is a natural isomorphism 
\[
f(X\otimes Y)\cong fX\otimes fY\quad\text{for $X,Y$ in $\cat T$.}
\]
The condition on an adjoint pair $(f,g)$ that both functors preserve products and coproducts arises naturally and has been studied in the context of tensor triangulated categories and Grothendieck--Neeman duality; see \cite[Theorem~1.7]{Balmer/DellAmbrogio/Sanders:2016a}.

\begin{proposition}
\label{pr:regularity-tt} 
Let $f\colon \cat T\to\cat U$ be an exact functor between rigidly compactly generated tensor triangulated categories that is strong symmetric monoidal and preserves coproducts. Suppose also that following conditions hold:
\begin{enumerate}[\quad\rm(1)] 
\item $\one_{\cat T}\in \thick(g(\one_{\cat U}))$;
\item the canonical map $\End^*_{\cat T}(\one)\to \End^*_{\cat U}(\one)$ is finite.
\end{enumerate}
If $\cat U$ is locally regular, then so is $\cat T$. 
\end{proposition}

\begin{proof}
Set $R=\End^*_{\cat T}(\one)$ and $S=\End^*_{\cat U}(\one)$, and view $\cat U$ as an $R$-linear category via the natural map $R\to S$.  We claim that, under the hypothesis of the statement, when the ring $S$ is noetherian so is $R$.

Indeed, assume the ring $S$ is noetherian. Since  $S$ is finite as an $R$-module, for  $I=\Ker(R\to S)$, the ring $R/I$ is noetherian, by the theorem of Eakin and Nagata \cite[Theorem~3.7]{Matsumura:1986a}. Thus $S$ is noetherian as an $R$-module. Then, since the adjunction isomorphism $\Hom^*_{\cat T}(\one_{\cat T}, g(\one_{\cat U})) \cong \Hom^*_{\cat U}(\one_{\cat U}, \one_{\cat U})=S$ is $R$-linear, the $R$-module $\Hom^*_{\cat T}(\one_{\cat T}, g(\one_{\cat U}))$ is noetherian. Since $\one_{\cat T}$ is in the thick subcategory generated by $g(\one_{\cat U})$, the $R$-module $\Hom^*_{\cat T}(\one_{\cat T},\one_{\cat T})=R$ is noetherian, as claimed.

Suppose $\cat U$ is locally regular over $S$; in particular $S$, and hence also $R$, is noetherian. The map $R\to S$ is finite, so Lemma~\ref{le:change-of-rings} yields that $\cat U$ is also locally regular over $R$. We claim  Proposition~\ref{pr:regularity} applies, so $\cat T$ is locally regular over $R$.

Indeed, from \cite[Theorem~1.3]{Balmer/DellAmbrogio/Sanders:2016a} we get that $f$ and $g$ both preserve products and coproducts. For $X$ in $\cat T$ the projection formula from \cite[Theorem~1.3]{Balmer/DellAmbrogio/Sanders:2016a} yields  
    \[
    X\otimes gf(\one_{\cat T})\cong g(fX\otimes f\one_{\cat T}) \cong g(f(X\otimes \one_{\cat T}))\cong gfX\,.
    \]
By hypothesis there exists an integer $d$ such that $\one_{\cat T}$ is in $\thick^d(gf\one_{\cat T})$. Applying $X\otimes -$ and using the isomorphism above, we deduce that $X$ is in $\thick^d(gfX)$.
\end{proof}

The following example illustrates the result above. In particular it shows that $\cat T$ locally regular need not imply $\cat U$ locally regular.

\begin{example}
Let $\varphi\colon A\to B$ be a finite map of commutative noetherian rings such that $B$ is faithful and of finite projective dimension as an $A$-module.
Then the adjoint pair
\[
\begin{tikzcd}[column sep = large]
    \bfD(A)\arrow[yshift=.75ex]{rr}{\varphi^*=B\lotimes_A-}&&  \bfD(B)  \arrow[yshift=-.75ex]{ll}{\varphi_*=\mathrm{res}}
  \end{tikzcd}
\]
satisfies the hypothesis of Proposition~\ref{pr:regularity-tt}. Indeed, it is clear that $\varphi^*$ is strong symmetric monoidal, preserving coproducts. Since $\varphi_* B$ is a perfect complex and has full support, the theorem of Hopkins~\cite{Hopkins:1987a} and Neeman~\cite{Neeman:1992a} yields that $A$ is in $\thick(\varphi_* B)$, so condition (1) holds. Condition (2) holds because the endomorphisms of the units in $\dcat A$ and $\dcat B$ are $A$ and $B$, respectively.  Proposition~\ref{pr:regularity-tt} subsumes the well known result that when $B$ is regular so is $A$. The converse need not hold; consider the case $A$ is a field.
\end{example}

\subsection*{Dualisable torsion objects}
Consider for  $\fp$ in $\Spec R$ the category of $\fp$-local $\fp$-torsion objects and describe its dualisable objects.

\begin{proposition}
\label{pr:regular-gorenstein}
    Let $\cat T$ be a rigidly compactly generated tensor triangulated category that is locally regular. For $\fp$ in $\Spec R$ and any object $X$ in $\gam_\fp{\cat T}$ the following conditions are equivalent.
    \begin{enumerate}[\quad\rm(1)]
        \item $X$ is dualisable in $\gam_\fp\cat T$;
        \item $\Hom^*_{\cat T}(C,X)$ is artinian over $R_\fp$ for each compact object $C$ in $\cat T$;
        \item $\Hom^*_{\cat T}(C,X)$ has finite length over $R_\fp$ for each compact object $C$ in $\gam_{\fp}\cat T$;
        \item $\Hom^*_{\cat T}(X,C)$ has finite length over $R_\fp$ for each compact object $C$ in $\gam_{\fp}\cat T$;
        \item $X$ is in $\thick(\gam_\fp\comp{\cat T})$;
        \item $X\to(D\phat)^2 X$ is an isomorphism.
    \end{enumerate}
\end{proposition}

\begin{proof} The proof combines various results from previous sections; only the implication (2)$\Rightarrow$(5) uses the local regularity condition.

    (1)$\Rightarrow$(2): This follows from Proposition~\ref{pr:dualising-hom}.
    
    (2)$\Leftrightarrow$(3): This follows from Proposition~\ref{pr:artinian}. 

(3)$\Leftrightarrow$(4): This follows from Lemma~\ref{le:hom-symmetry}. 

(2)$\Rightarrow$(5): This follows from Proposition~\ref{pr:strong-regular}.

(5)$\Rightarrow$(1): This follows from Corollary~\ref{co:plocal-cat}.

(2)$\Leftrightarrow$(6): For any $C\in\comp{\cat T}$ the $R\phat$-module $\Hom^*_{\cat T}(C,X)$ is $\fp$-torsion, so the  assertion follows from the analogue of Lemma~\ref{le:BC} for $D\phat$ and Lemma~\ref{le:reflexive-art} below.
\end{proof}

\begin{lemma}
\label{le:reflexive-art}
    Let $M$ be an $R\phat$-module that is $\fp$-torsion. Then the canonical map $\delta_M\colon M\to(D\phat)^2 M$ is an isomorphism if and only if $M$ is artinian.
\end{lemma}

\begin{proof}
    From Matlis duality it follows that the map $\delta_M$ is an isomorphism if $M$ is artinian. Suppose $\delta_M$ is an isomorphism. An application of the Snake Lemma and the fact that $\delta_M$ is always a monomorphism shows that Matlis reflexive modules are closed under submodules and quotients. 
    In particular, the socle of $M$ has finite length. This is an essential submodule since $M$ is $\fp$-torsion, and therefore $M$ is artinian.
\end{proof}

\subsection*{Compactness}
Given $\fp\in\Spec R$, we consider in $\gam_\fp\cat T$ the difference between compact and dualisable objects; in cohomological terms it is the difference between having finite length  and being artinian.

\begin{corollary}\label{co:compact-dualisable}
Let $\cat T$ be as in Proposition~\ref{pr:regular-gorenstein}. For $\fp$ in $\Spec R$ and any object $X$ in $\gam_\fp{\cat T}$ the following conditions are equivalent.
\begin{enumerate}[\quad\rm(1)]
    \item $X$ is compact in $\gam_{\fp}{\cat T}$;
    \item $\Hom^*_{\cat T}(C,X)$ has finite length over $R_\fp$ for each compact object $C$ in ${\cat T}_\fp$;
    \item $\Hom^*_{\cat T}(X,C)$ has finite length over $R_\fp$ for each compact object $C$ in ${\cat T}_\fp$.
\end{enumerate}
In particular, compact objects and dualisable objects in $\gam_\fp\cat T$ coincide if and only if $\fp$ is minimal in $\supp_R\cat T$
\end{corollary}

\begin{proof}
(1)$\Rightarrow$(2): The $R_\fp$-module $\Hom^*_{\cat T}(C,X)$ is noetherian, because $X$ is also compact in $\cat T_\fp$, and $\fp$-torsion, so it has finite length.

(2)$\Leftrightarrow$(3): This follows from Lemma~\ref{le:hom-symmetry}. 

(3)$\Rightarrow$(1): Passing to $\cat T_\fp$ and $R_\fp$ we can assume $R$ is local. The hypothesis implies $X$ is dualisable in $\gam_\fp \cat T$, by 
Proposition~\ref{pr:regular-gorenstein}. Thus $\kos{X}{\fp}\cong X\otimes\kos{\one}{\fp}$ is compact in $\gam_\fp\cat T$, by Proposition~\ref{pr:dualisable}, and hence also in $\cat T$. We claim $\thick(X)=\thick(\kos X{\fp})$, which implies $X$ is compact as desired.

For the claim it suffices to show that $\End^*_{\cat T}(X)$ is $\fp$-torsion, by \cite[Lemma~3.9]{Benson/Iyengar/Krause:2015a}.
First observe that for any object $C$ there is an isomorphism of $R$-modules
\[
\Hom^*_{\cat T}(X,\gam_\fp C) \cong \Hom^*_{\cat T}(X,C)
\]
since $X$ is $\fp$-torsion. Thus the hypothesis implies that the $R$-module $\Hom^*_{\cat T}(X,\gam_\fp C)$ has finite length for each compact object $C$ in ${\cat T}_\fp$. Since $X$ is dualisable in $\gam_\fp\cat T$, it is finitely built from objects of the form $\gam_\fp C$; this is again from Proposition~\ref{pr:regular-gorenstein}. Thus $\Hom^*_{\cat T}(X,X)$ has finite length, and in particular it is $\fp$-torsion.

The last part of the statement follows from Lemma~\ref{le:minimal}, using the cohomological descriptions of compact and dualisable objects given above.
\end{proof}

\subsection*{Dualisable complete objects}
 There is an equivalence $\gam_\fp\cat T\iso\lam^\fp\cat T$ of compactly generated tensor triangulated categories, for $\fp$ in $\Spec R$. Clearly, this equivalence preserves and reflects dualisability, and restricts to a triangle equivalence 
\[
\thick(\gam_\fp\comp{\cat T})\iso\thick(\lam^\fp\comp{\cat T}_\fp)\,.
\] 
This leads to an analogue of Proposition~\ref{pr:regular-gorenstein} for $\lam^\fp\cat T$.

\begin{proposition}
\label{pr:regular-gorenstein-complete}
    Let $\cat T$ be a rigidly compactly generated tensor triangulated category that is locally regular. For $\fp$ in $\Spec R$ and any object $X$ in $\lam^\fp{\cat T}$ the following conditions are equivalent.
    \begin{enumerate}[\quad\rm(1)]
        \item $X$ is dualisable;
        \item $\Hom^*_{\cat T}(C,X)$ is noetherian over $R\phat$ for each compact object $C$ in $\cat T$;
        \item $\Hom^*_{\cat T}(C,X)$ has finite length over $R\phat$ for each compact object $C$ in $\lam^{\fp}\cat T$;
        \item $\Hom^*_{\cat T}(X,C)$ has finite length over $R\phat$ for each compact object $C$ in $\lam^{\fp}\cat T$;
        \item $X$ is in $\thick(\lam^\fp\comp{\cat T}_\fp)$;
        \item $X\to(D\phat)^2 X$ is an isomorphism.
    \end{enumerate}
\end{proposition}

\begin{proof}
Given the equivalence $\gam_\fp\cat T\iso\lam^\fp\cat T$, the result follows from Proposition~\ref{pr:regular-gorenstein}. The only condition that does not translate directly is (2), but this is equivalent to (3) by Proposition~\ref{pr:noetherian}.     
   \end{proof}

\begin{corollary}
  Let $\cat T$ be a rigidly compactly generated tensor triangulated category that is locally regular. For $\fp$ in $\Spec R$  the functor $D\phat$ restricts to a pair of mutually quasi-inverse contravariant equivalences $\cat A_\fp\rightleftarrows\cat N_\fp$, identifying the dualisable objects in $\gam_{\fp}\cat T$ with the dualisable objects in $\lam^{\fp}\cat T$. Hence the subcategories of compact objects in $\gam_{\fp}\cat T$ and $\lam^{\fp}\cat T$  agree and equal  $\cat A_\fp\cap\cat N_\fp$.
\end{corollary}

\begin{proof}
    The pair of equivalences  $\cat A_\fp\rightleftarrows\cat N_\fp$ has been established in Theorem~\ref{th:bc-duality}. The statement about dualisable objects then follows from Propositions~\ref{pr:regular-gorenstein} and \ref{pr:regular-gorenstein-complete}. The statement about compact objects follows from Corollary~\ref{co:compact-dualisable}.
\end{proof}

\subsection*{Strong generation}
Here is a consequence of Proposition~\ref{pr:regular-gorenstein} and Lemma~\ref{le:reg-bcd+dg}.

\begin{corollary}
\label{co:kinj-main}
For each $\fp$ in $\Spec R$, the thick subcategory of dualisable objects in $\gam_\fp \cat T$ and in $\lam^\fp\cat T$ have strong generators. \qed
\end{corollary}

\section{Finite groups}
\label{se:finite-groups}
The main result in this section is that for any finite group $G$, the tensor triangulated category $\KInj{kG}$ is locally regular. Applications include the various characterisations of local dualisable objects stated in the introduction. With an eye towards future applications, we establish an intermediate step in the more general setting of finite group schemes. We begin with recalling some fundamental results about cohomology and representations of finite group schemes.

\subsection*{Finite dimensional Hopf algebras}
Let $k$ be a field of positive characteristic and $A$ a finite dimension co-commutative Hopf algebra over $k$.
Equivalently, $A \cong kG$ is a group algebra of some finite group scheme $G$ over $k$, that is to say, the $k$-linear dual to the coordinate algebra $k[G]$ of $G$.  The homotopy category $\KInj A$ of injective $A$-modules is a compactly generated triangulated category, and the quotient functor $q\colon \KInj A\to \dcat{\Mod A}$ induces an equivalence
\begin{equation}
\label{eq:kinj-compact}
 \comp{\KInj A} \longiso \dbcat A\,;   
\end{equation}
its quasi-inverse sends an $A$-complex $M$ to its injective resolution, $\bfi M$; see \cite{Benson/Iyengar/Krause/Pevtsova:2018a}.

Endowed with the product given by $X\otimes_k Y$ with diagonal $A$-action, $\KInj A$ is a tensor triangulated category, with unit  $\bfi_A k$, where $k$ is viewed as an $A$-module via the augmentation $A\to k$. Moreover $\KInj A$ is rigidly compactly generated. For each compact object $C$ in $\KInj A$, the $R=\Ext^*_A(k,k)$-module $\Ext_A^*(C,C)$ is noetherian, by \cite[Theorem 1.1]{Friedlander/Suslin:1997a}. In particular, the tensor triangulated category $\KInj A$ is noetherian. 

In the next theorem we show for any finite commutative unipotent group scheme, its homotopy category of injectives is locally regular. 

\begin{theorem}
\label{th:hopf}
Let $k$ be a field of positive characteristic $p$ and $A$ a finite dimensional co-commutative Hopf algebra over $k$. Assume that there exist integers $q_i=p^{d_i}$ for  $1\le i\le r$ such that there is an isomorphism of $k$-algebras
\[
A\cong \frac{k[z_1,\dots,z_r]}{(z_1^{q_1},\dots,z_r^{q_r})}\,.
\]
 Then the tensor triangulated category $\KInj A$ is locally regular.
\end{theorem}

\begin{proof}
Set $\cat T=\KInj{A}$ and $R=\End^*_{\cat T}(\bfi_A k)\cong \Ext^*_A(k,k)$.  Since $A$ is local, its residue field $k$ generates $\dbcat{A}$ as a thick subcategory, so $\bfi_A k$ is a compact generator $\cat T$, by the equivalence \eqref{eq:kinj-compact}. Thus it suffices to prove that $\thick(\gam_\fp(\bfi_A k))$ has a strong generator for each $\fp$ in $\Spec R$. 

As is well-understood, the object $\gam_\fp(\bfi_A k)$ is independent of the coproduct on $A$; see \cite[Propostion~4.9]{Benson/Iyengar/Krause:2011a}. Thus the same is true of the local regularity condition. In what follows we assume that the coproduct on $A$ is the one for which each $z_i$ is primitive: $z_i\mapsto z_i\otimes 1 + 1\otimes z_i$. Let $B$ be the Koszul complex of $A$ on elements $z_1,\dots,z_r$, where $B$ is viewed as a dg (=differential graded) Hopf $A$-algebra, as in \cite[\S7]{Benson/Iyengar/Krause:2011b}; this is where the specific choice of the coproduct on $A$ is needed. Let $\cat U$ be the homotopy category of graded-injective dg $B$-modules; see \cite[\S4]{Benson/Iyengar/Krause:2011b}. Let $\bfi_B k$ be the semi-injective resolution of $k$ viewed as a dg $B$-module via the augmentation $B\to k$, and set $S=\End^*_{\cat U}(\bfi_B k)$; this is the polynomial ring on $r$ indeterminates over $k$, of degree $2$, see \cite[Lemma~7.1]{Benson/Iyengar/Krause:2011b}. The triangulated category $\cat U$ is compactly generated by $\bfi_B k$  and it is $S$-linear; see \cite[\S4 and \S6]{Benson/Iyengar/Krause:2011b} for details. The global dimension of $S$ is $r$ and so finite. Then, keeping in mind Lemma~\ref{le:reg-bcd+dg}, we apply Theorem~\ref{th:reg-end} to conclude that $\cat U$ is locally regular over $S$. 

 Consider the structure map $A\to B$; this is a map of dg Hopf algebras, where $A$ is viewed as a dg algebra with zero differential. As an $A$-module, $B$ is finite free so one gets an exact functor $f^*\coloneqq B\otimes_{A}-\colon \cat T\to \cat U$ with right adjoint $f_*$ induced by restriction along $A\to B$. In particular $f_*(\bfi_B k)$ is an injective resolution of $k$ as an $A$-module, and hence a compact generator of $\cat T$. Note that $R=\End^*_{\cat T}(f_*(\bfi_B k))$ and let $\varphi\colon S\to R$ be the natural map induced by $f_*$, and view $\cat T$ as an $S$-linear category. Since $A\to B$ is a map of Hopf algebras, the functors $f^*$ and $f_*$ are $S$-linear, and satisfy the hypotheses of  Proposition~\ref{pr:regularity}.

Indeed, one has an isomorphism of dg $B$-modules $\Hom_A(B,A)\cong \Sigma^{-r}B$, see \cite[Proposition~1.6.10]{Bruns/Herzog:1998a}. Thus $\Sigma^{-r}f^*$ is right adjoint to $f_*$,  and both $f^*$ and $f_*$ preserve products and coproducts.  Since $A$ is a commutative artinian local ring and $B$ is the Koszul complex on a finite set of elements in $A$, one has $A\in \thick^d(B)$ for some integer $d$; see \cite[Lemma~6.0.9]{Hovey/Palmieri/Strickland:1997a}. Thus for $X\in \cat T$ we deduce that
\[
f_*f^*X= B\otimes_{A} X \in \thick^d(A\otimes_{A}X) =  \thick^d(X)\,.
\]
Since we already know that $\cat U$ is locally regular over $R$, Proposition~\ref{pr:regularity} applies and yields that $\cat T$ is locally regular as an $S$-linear category.
The map $\varphi$ is finite so  Lemma~\ref{le:change-of-rings} implies that $\cat T$ is also locally regular as an $R$-linear category; in other words, it is locally regular as a tensor triangulated category.  
\end{proof}

\subsection*{Finite groups}
Let $G$ be a finite group and $k$ a field whose characteristic is positive and divides the order of $G$. Let $kG$ be the group algebra; this is a finite dimensional co-commutative Hopf algebra over $k$, so $\KInj{kG}$ has a natural structure as a tensor triangulated category, and this is rigidly compactly generated.

\begin{theorem}
\label{th:regularity-main}
The tensor triangulated category $\KInj{kG}$ is locally regular. 
\end{theorem}

\begin{proof}
We can assume the characteristic of $k$, say $p$, divides the order of $G$. Consider first the case where $G=E$ is an elementary abelian $p$-group of rank $r$. Its group algebra is of the form
\[
kE\cong \frac{k[z_1,\dots,z_r]}{(z_1^p,\dots,z_r^p)}
\]
Thus Theorem~\ref{th:hopf} applies, yielding that $\KInj{kE}$ is locally regular as a tensor triangulated category.  We deduce the result for an arbitrary $G$ using Proposition~\ref{pr:regularity-tt}.

Let $U$ be the (finite) collection of maximal elementary abelian $p$-subgroups of $G$. For each $E\in U$ we know that the tensor triangulated category $\KInj{kE}$ is locally regular, and this implies that the tensor triangulated category
 \[
\cat U\coloneqq \prod_{E\in U} \KInj{kE}
 \]
is locally regular as well. For a maximal elementary abelian $p$-subgroup $E\subset G$ let 
\[
f^*_E\colon \KInj{kG}\to \KInj{kE}\quad\text{and}\quad f_*^E\colon \KInj{kE}\to \KInj{kG}
\]
denote the restriction functor and induction functor, $kG\otimes_{kE}-$, respectively. Observe  that $(f^*_E,f_*^E,f^*_E)$ is an adjoint triple. Consider functors
\begin{alignat*}{2}
f^* \colon &\KInj{kG} \longrightarrow \prod_{E\in U} \KInj{kE}& 
        \quad &{\text{where } X\mapsto (f^*_EX)}\,, \\
f_* \colon  &\prod_{E\in U} \KInj{kE}  \longrightarrow \KInj{kG}& 
        \quad &{\text{where } (Y_E) \mapsto \bigoplus_{E\in U} f^E_*(Y_E)} \,.
\end{alignat*}
We claim that the adjoint pair $(f^*,f_*)$ satisfies the hypotheses of Proposition~\ref{pr:regularity-tt}.

Indeed, it is clear that each $f^*_E$, and hence $f^*$, is strong symmetric monoidal and preserves coproducts.  The unit of the tensor triangulated category $\cat U$ is $\prod_{E\in U}\bfi_E k$, where $\bfi_E k$ is the injective resolution of $k$ as a $kE$-module. It follows from \cite[Theorem~2.1]{Carlson:2000c} that $k$ is in the thick subcategory of $\dbcat{kG}$ generated by $\{f^E_*(k)\}_{E\in U}$, that is to say that,  $\bfi_G k$ is in $\thick(f_*(\one_{\cat U}))$.

The graded endomorphism rings of $\one_{\cat T}$ and $\one_{\cat U}$ are $H^*(G,k)$ and $\prod_{E\in U}H^*(E,k)$, respectively. For each $E$ the natural map $H^*(G,k)\to H^*(E,k)$ obtained by restriction is finite, so the following map  is finite as well:
\[
H^*(G,k)\to \prod_{E\in U}H^*(E,k)\,.
\]
Thus Proposition~\ref{pr:regularity-tt} applies to yield that $\KInj{kG}$ is locally regular.
\end{proof}

\subsection*{A recollement at the irrelevant ideal}
Set $\fm= H^{\geqslant 1}(G,k)$, the closed point of $\Spec H^*(G,k)$. Then the recollement  \eqref{eq:recoll-spcl} given by $V=\{\fm\}$ reads:
\[
\begin{tikzcd}[column sep = huge]
    \StMod{kG} \arrow[tail]{r} 
    	&\KInj{kG} \arrow[twoheadrightarrow,swap,yshift=1.5ex]{l}    \arrow[twoheadrightarrow,yshift=-1.5ex]{l}
		    \arrow[twoheadrightarrow]{r}[description]{q} 
        &\dcat{\Mod kG}\arrow[tail,swap,yshift=1.5ex]{l}{q_\lambda}\arrow[tail,yshift=-1.5ex]{l}{q_\rho}
  \end{tikzcd}
\]
We have $\gam_\fm=q_\lambda\circ q$, and therefore $q$ induces an equivalence 
\[
\gam_\fm \KInj{kG}\iso \dcat{\Mod kG}\,.
\]
The left adjoint $q_\lambda$ sends $M$ in $\dbcat{kG}$ to $\mathbf{p}M$, the projective resolution of $M$, which is $\gam_\fm(\bfi M)$; see~\cite[\S6]{Benson/Krause:2008a}. In particular, $q$ restricts to an equivalence 
\begin{equation}\label{eq:irrelevant}
\gam_\fm(\comp{\KInj{kG}})\iso \dbcat{kG}\,.
\end{equation}
The triangulated category $\dbcat{kG}$ has a strong generator, namely, the direct sum of representatives of isomorphism classes of simple $kG$-modules, viewed as a complex concentrated in a single degree. Thus if $M$ is a generator of $\dbcat{kG}$, the complex $\gam_\fm(\bfi M)$ is a strong generator for $\thick(\gam_\fm(\bfi M))$.

\subsection*{Local dualisable objects}

Putting together the result of the previous sections we are able to describe the dualisable objects for each point. The formulation of the following result requires the choice of a compact generator, but it is clear that it does not depend on the choice. A natural candidate is the injective resolution of the direct sum of representatives of isomorphism classes of simple $kG$-modules.

\begin{theorem}
\label{th:kinj-main}
Let $G$ be a finite group and $k$ a field. Set $\cat T=\KInj{kG}$ and $R= H^*(G,k)$.  Fix a compact generator $C$ for $\cat T$ and a point $\fp$ in $\Spec R$. For each $X$ in $\gam_{\fp} {\cat T}$, the following conditions are equivalent.
\begin{enumerate}[\quad\rm(1)]
\item $X$ is dualisable in $\gam_\fp\cat T$;
\item $\bcd X$ is dualisable in $\lam^\fp\cat T$;
\item $X$ is in $\thick(\gam_\fp C_\fp)\subseteq \gam_\fp\cat T$;
\item $D_\fp X$ is in $\thick(\Lambda^\fp C_\fp)\subseteq\Lambda^\fp\cat T$;
\item $\Hom^*_{\cat T}(C,X)$ is artinian over $R_\fp$;
\item $\Hom^*_{\cat T}(C,D_\fp X)$ is noetherian over $R\phat$;
\item $\Hom^*_{\cat T}(\kos{C}{\fp},X)$ has finite length over $R_\fp$;
\item $X\to(D\phat)^2 X$ is an isomorphism.
\end{enumerate} 
\end{theorem}

\begin{proof}
As noted before, the tensor triangulated category $\cat T$ is rigidly compactly generated. It is also locally regular, by Theorem~\ref{th:regularity-main}. Given this, the equivalences (1)$\Leftrightarrow$(3)$\Leftrightarrow$(5)$\Leftrightarrow$(7)$\Leftrightarrow$(8) are contained in Proposition~\ref{pr:regular-gorenstein}. Because the Brown--Comenetz dual $\bcd X$ is in $\lam^\fp\cat T$, the equivalence (5)$\Leftrightarrow$(6) follows from Theorem~\ref{th:bc-duality}. Finally, the equivalences (2)$\Leftrightarrow$(4)$\Leftrightarrow$(6) are contained in Proposition~\ref{pr:regular-gorenstein-complete}.
\end{proof}

\subsection*{The stable module category}
The stable module category, $\StMod{kG}$ is not noetherian in general, for the endomorphism of its unit is the Tate cohomology ring of $kG$. Thus we cannot apply the results of the previous sections directly to this category. However, one can view $\StMod{kG}$ as an $R=H^*(G,k)$-linear category, and then its support consists of all $\fp\in \Spec R$, except the maximal ideal $H^{\geqslant 1}(G,k)$, and for such $\fp$ the embedding $\StMod{kG}\to \KInj{kG}$ via the category of acyclic complexes induces an equivalence of tensor triangulated categories
\[
(\StMod{kG})_\fp \simeq {\KInj{kG}}_\fp\,;
\]
see \cite[Lemma~2.6]{Benson/Iyengar/Krause/Pevtsova:2019a}. Since all the conditions in Theorem~\ref{th:kinj-main} are local at $\fp$, one gets Theorem~\ref{th:finitelength} stated in the introduction.

\subsection*{Pure-injectives} A module  is \emph{pure-injective}, or \emph{algebraically compact} if it is injective with respect to pure-exact sequences, where an exact sequence is \emph{pure-exact} if it remains exact after tensoring with any module. A module is \emph{$\Sigma$-pure-injective} if any coproduct of
copies of it is pure-injective. There are many equivalent formulations, see for example \cite[\S4.4.2]{Prest:2009a}.

\begin{corollary}
  The dualisable modules in $\gam_\fp\StMod kG$ are  $\Sigma$-pure-injective.
\end{corollary}

\begin{proof}
This follows from Proposition~\ref{pr:pure-inj} once we observe that under the canonical functor $\Mod kG\to\StMod kG$ a $kG$-module is pure-injective if and only if it is pure-injective as an object of the triangulated category $\StMod kG$.
\end{proof}

\begin{example}
  Set $G=\bbZ/2\times\bbZ/2$ and let $k$ be a field of characteristic two. Set $\cat T=\StMod kG$. Then one has 
\[ 
H^*(G,k)=\Ext^*_{kG}(k,k)\cong k[\zeta_1,\zeta_2] 
\]
with $\deg(\zeta_1)=1=\deg(\zeta_2)$. Let $\fp\in\Proj H^*(G,k)\cong\bbP^1_k$ be a closed point and choose a homogeneous irreducible element $\zeta$ of degree $d$  generating $\fp$.  The bijection
\[
\uHom_{kG}(k,\Omega^{-r}(k)) \longiso \Ext^r_{kG}(k,k)\qquad r\ge 1
\]
gives for each power $\zeta^n$ a monomorphism $k\to\Omega^{-nd}(k)$ whose cokernel is $\kos k{\fp^n}$. It is just a shift of the Carlson module associated to $\zeta^n$. Thus for each $n$ there is an exact sequence
\[
0\lra k\lra \Omega^{-nd}(k)\lra \kos k{\fp^n} \lra 0
\]
and multiplication by $\zeta$ induces monomorphisms $\kos k{\fp^1} \rightarrowtail \kos k{\fp^2}\rightarrowtail\cdots$ with
\[
\kos k{\fp^\infty}\coloneqq \colim_{n\ge 1} \kos k{\fp^n}\cong \gam_\fp k\,.
\]
Applying Spanier--Whitehead duality $\Hom_k(-,k)$ yields a sequence of epimorphisms $\cdots \twoheadrightarrow \kos k{\fp^2} \twoheadrightarrow \kos k{\fp^1}$
with
\[
\lim_{n\ge 1} \kos k{\fp^n} \cong \lam^\fp k\,.
\]
For a class $\cat C$ of objects in an additive category let $\add\cat C$ denote the full subcategory consisting of direct summands of finite direct sums of objects in $\cat C$. Then we have 
\[
\comp{(\gam_\fp\cat T)}=\add\{ \kos k{\fp^n} \mid 1\le  n<\infty\}=\comp{(\lam^\fp\cat T)}\,.
\]
On the other hand there are more dualisable objects. Let $\cat S^{\mathrm d}$ denote the subcategory of dualisable objects in a tensor triangulated category $\cat S$.
 \begin{align*}
& (\gam_\fp\cat T)^{\mathrm d}=\add\{ \kos k{\fp^n} \mid 1\le n\le\infty\} \\
 &(\lam^\fp\cat T)^{\mathrm d}=\add(\{ \kos k{\fp^n} \mid 1\le n<\infty\}\cup\{\lam^\fp k\})\,.
 \end{align*}
\end{example}

\subsection*{Balmer spectrum}
The dualisable objects in $\gam_\fp\KInj{kG}$ form an essentially small tensor triangulated category which we denote $(\gam_\fp\KInj{kG})^{\mathrm d}$.
So one may wonder about its associated Balmer spectrum~\cite{Balmer:2005a}.

\begin{question}
Let $G$ be a finite group and $\fp$ in $\Spec H^*(G,k)$. For the tensor unit $\one=\gam_\fp(\bfi k)$ in $(\gam_\fp\KInj{kG})^{\mathrm d}$ we have an isomorphism $\End^*(\one)\cong H^*(G,k)\phat$, by \cite[Corollary~5.4]{Benson/Iyengar/Krause/Pevtsova:2019a}. 
Thus there is a canonical map
\[
\bSpec (\gam_\fp\KInj{kG})^{\mathrm d} \lra \Spec H^*(G,k)\phat
\] 
which is continuous and surjective \cite{Balmer:2010a}. When is this map a homeomorphism?
We do not know if it holds for general $\fp$, but we can answer this in the affirmative when $G$ is an elementary abelian $p$-group, using a reduction to the case of the derived category of a commutative ring. For a general group $G$, the question has an affirmative answer when $\fp=H^{\geqslant 1}(G,k)$, for the category of dualisable objects at this prime is equivalent to $\dbcat{kG}$ via \eqref{eq:irrelevant}. In this case the homeomorphism above is established in \cite{Benson/Carlson/Rickard:1995a,Benson/Iyengar/Krause:2011b}, using that $H^*(G,k)\cong H^*(G,k)\phat$ for this $\fp$.
\end{question}

\section{\textpi-points} 
\label{se:pipoints}
The main result in this section is the following characterisation of the finite length property in the stable module category of a finite group scheme in terms of the restrictions to $\pi$-points.  The new notions and constructions that appear in the statement, including that of a $\pi$-point, are recalled further below.
 
\begin{theorem}
\label{th:pi-main} 
Let $k$ be a field of positive characteristic $p$ and $G$ a finite group scheme over $k$.
Fix $\fp$ in $\Proj H^*(G,k)$ and a $\pi$-point $\alpha\colon K[t]/(t^p) \to KG$ for $\fp$. Let $\fm$ be the corresponding closed point in $\Proj H^*(G,K)$. For any $M$ in $\gam_\fp \StMod kG$ the following conditions are equivalent.
\begin{enumerate}[\quad\rm(1)] 
\item $M$ has finite length in $\gam_\fp\StMod kG$;
\item $\gam_\fm (M_K)$ finite length in $\gam_\fm\StMod KG$;
\item $\alpha^*(M_K)$ is in $\stmod K[t]/(t^p)$.
\end{enumerate} 
\end{theorem}

Condition (3) is the statement that $K\otimes_kM$ is stably finite dimensional when restricted to $K[t]/(t^p)$; that is to say, isomorphic to a direct sum of a finite dimensional module and a free module. A corollary is that condition (3) is independent of the choice of a $\pi$-point representing $\fp$. 
When $G$ is a finite group, conditions (1) and (2) are equivalent to dualisability of the relevant modules; see Theorem~\ref{th:kinj-main}.

\subsection*{\textpi-points} 
We recall some facts about $\pi$-points; see \cite{Friedlander/Pevtsova:2007a},   \cite{Benson/Iyengar/Krause/Pevtsova:2018a}, \cite{Benson/Iyengar/Krause/Pevtsova:2017a} for details. Let $k$ be a field of positive characteristic $p$ and $G$ a finite group scheme; thus $kG$, the group algebra of $G$, is a finite dimensional co-commutative Hopf algebra over $k$. A $\pi$-point of $G$ is a flat map of $K$-algebras 
\[
\alpha \colon K[t]/(t^p) \longrightarrow KG
\]
defined over some field extension $K/k$ and factoring through the group algebra of a unipotent abelian subgroup scheme of $G_K$; when $G$ is a finite group, this means that $\alpha$ factors through $KH$ where $H$ is an  abelian $p$-subgroup of $G$. One gets adjoint functors
\[
\begin{tikzcd}[column sep = huge]
   \StMod K[t]/(t^p)  \arrow[yshift=.75ex]{r}{\alpha_*}  
        &\StMod KG \arrow[yshift=-.75ex]{l}{\alpha^*}    
            \arrow[swap,yshift=-.75ex]{r}{\downarrow_{kG}}
        &\StMod kG \arrow[swap,yshift=.75ex]{l}{K\otimes_k-}\,.
    \end{tikzcd}
\]
where $\alpha_*\coloneqq KG\otimes_{K[t]/(t^p)}-$ and $\alpha^*$ is restriction along $\alpha$, and $\downarrow_{kG}$ is restriction along $kG\to KG$. For a $kG$-module $M$ set $M_K=K\otimes_k M$.

A $\pi$-point $\alpha\colon K[t]/(t^p) \to KG$ induces a map in cohomology 
\[
H^*(G,k)\coloneqq \Ext^*_{kG}(k,k) \longrightarrow \Ext^*_{K[t]/(t^p)}(K,K)
\]
whose kernel yields a point on $\Proj H^*(G,k)$. Two $\pi$-points are \emph{equivalent} if they correspond to the same homogeneous prime ideal $\fp$. This gives a one-to-one correspondence between equivalence classes of $\pi$-points and  points in $\Proj H^*(G,k)$, satisfying the following condition:
for any $\pi$-point $\alpha$ in the equivalence class corresponding to $\fp$ and 
for any $kG$-module $M$, one has
\[
\fp \in \supp(M) \Longleftrightarrow \alpha^*(M_K) \ne 0 \text{ in } \StMod K[t]/(t^p)
\] 
where $\supp(M)$ equals the set of points  $\fp\in\Proj H^*(G,k)$ such that $\gam_\fp M\neq 0$.

\subsection*{Finite length}
Our goal is a characterisation of finite length objects in the local category  $\gam_\fp \StMod kG$ via $\pi$-points.
We begin with a simple observation.

\begin{lemma} 
\label{le:cyclic} 
Let $K$ be a field of positive characteristic $p$ and $A\coloneqq K[t]/(t^p)$. Let $\fm$ be the unique point in $\Proj \Ext^*_{A}(K,K)$.  An $A$-module $N$ is in $\stmod A$ if and only if 
the $\Ext^*_{A}(K,K)_\fm$-module $\Ext^*_{A}(K,N)_\fm$ has finite length. \qed
\end{lemma}

We require also the following observation. Recall the convention about Koszul objects for graded modules, discussed above Lemma~\ref{le:koszul-tt}.

\begin{lemma}
\label{le:koszul-ca}
   Let $R\to S$ be a flat map of graded-commutative noetherian rings. Fix $\fp\in\Spec R$ and a point $\fq$ in $\Spec S$ lying over $\fp$. Let $X$ be an $R_\fp$-module such that $\fp^s X=0$ for some integer $s$. The following conditions are equivalent:
\begin{enumerate}[\quad\rm(1)]
    \item $\length_{R_\fp}X <\infty$;
    \item $\length_{S_{\fq}} (k(\fq)\otimes_R X)<\infty$;
    \item $\length_{S_{\fq}} H^*(\kos{(S\otimes_RX)_{\fq}}{\fq})<\infty$.
\end{enumerate}
\end{lemma}
   
\begin{proof}
We can assume $R$ and $S$ are local, with maximal ideals $\fp$ and $\fq$, respectively. 
Let $k$ and $l$ be the corresponding residue fields. 

 (1)$\Leftrightarrow$(2): Since $\fp^s X=0$, one has that
\begin{align*}
\length_R X<\infty 
    & \iff \length_k(k\otimes_RX)<\infty    \\
    & \iff \length_l(l\otimes_k(k\otimes_RX))<\infty \\
    & \iff \length_l(l\otimes_RX))<\infty \,.
\end{align*}
This is the desired equivalence.

(1)$\Rightarrow$(3): When $\length_RX$ is finite, the $S$-module $S\otimes_RX$ is finitely generated, and hence so is the $S$-module $H^*(\kos{(S\otimes_RX)}{\fq})$. Since the latter is also annihilated by $\fq$, it follows that it has finite length.

(3)$\Rightarrow$(2): 
One has isomorphisms of $S$-modules
\[
H^0(\kos{(S\otimes_RX)}{\fq})
    \cong \frac{(S\otimes_RX)}{\fq(S\otimes_RX)}
        \cong l\otimes_S(S\otimes_RX)
            \cong l\otimes_R X\,.
\]
This justifies the claim.
    \end{proof}

We are now ready to provide the proof of the main result of this section.

\begin{proof}[Proof of Theorem~\ref{th:pi-main}]
Let $\cat T = \StMod kG$ and $\cat T_K = \StMod KG$. Consider the $k$-algebra $R=H^*(G,k)$ and let 
\[
R\longrightarrow S=H^*(G_K,K)\cong K\otimes_k H^*(G,k)
\]
be the map of rings induced by $k\to K$. This is a flat map. 

(1)$\Leftrightarrow$(2): 
Let $C$ be the direct sum of a representative set of simple $kG$-modules. We provide a chain of equivalent conditions to get from (1) to (2). 

Since the object $C$ is a compact generator of $\cat T$,  the description of the compacts in $\gam_\fp\cat T$ from Lemma~\ref{le:p-torsion}, yields that that condition (1) is equivalent to:
\begin{equation*}
\label{eq:length}
\length_{R_\fp}\Hom_{\cat T}^*(\kos C\fp,M) < \infty\,.
\end{equation*}
Lemma~\ref{le:koszul-ca} applies to $\Hom_{\cat T}^*(\kos C\fp, M)$ since it is annihilated by $\fp^s$ for some $s$, again by Lemma~\ref{le:p-torsion}. Thus the above condition is equivalent to 
\begin{equation*}
\label{eq:length2}
\length_{S_\fm} H^*(\kos{(S \otimes_R \Hom_{\cat T}^*(\kos C\fp,M))_\fm}{\fm}) < \infty\,.
\end{equation*} 
The isomorphism of $S$-modules
\[
 S\otimes_R \Hom_{\cat T}^*(\kos C{\fp}, M) 
     \cong K\otimes_k \Hom_{\cat T}^*(\kos C{\fp}, M)
     \cong \Hom^*_{\cat T_K}((\kos C\fp)_K, M_K)
\] 
implies that the last condition is further equivalent to 
\begin{equation*}
\label{eq:length3}
\length_{S_\fm} H^*(\kos{\Hom^*_{\cat T_K}((\kos C\fp)_K, M_K)}{\fm})_\fm < \infty\,.
\end{equation*} 
Now by Lemma~\ref{le:koszul-tt}, this is equivalent to
\begin{equation*}
\label{eq:length4}
\length_{S_\fm} \Hom^*_{\cat T_K}((\kos C\fp)_K \otimes_K \kos K{\fm} , (M_K))_\fm <\infty\,.
\end{equation*} 
In $\cat T_K$ one has an isomorphism
\[
(\kos C\fp)_K \otimes_K \kos K{\fm} \cong \kos {C_K}{(\fp S+\fm)}\,.
\]
Since $\fp S\subseteq \fm$ the $KG$-modules $\kos {C_K}{(\fp S+\fm)}$ and $\kos {C_K}{\fm}$ generate the same thick subcategory of $\cat T_K$, and  hence the last condition is equivalent to
\begin{equation*}
\label{eq:length5}
\length_{S_\fm} \Hom^*_{\cat T_K}(\kos {C_K}\fm, M_K)_\fm<\infty\,.
\end{equation*} 
Finally, since $C_K$ generates $\cat T_K$ and one has an isomorphism
\[
\Hom^*_{\cat T_K}(\kos {C_K}\fm, M_K)_\fm\cong \Hom^*_{\cat T_K}(\kos {C_K}\fm, \gam_\fm (M_K))
\]
another application of Lemma~\ref{le:p-torsion} yields that the previous condition is equivalent to $\gam_\fm (M_K)$ having finite length in  $\gam_\fm\cat T_K$.

(2)$\Leftrightarrow$(3): 
From Lemma~\ref{le:compact-generator} below it follows that the object $\gam_\fm(M_K)$ is of finite length in $\gam_\fm\StMod KG$ if and only if
\[
\length_{S_\fm} \Hom^*_{\cat T_K}(\alpha_*K, M_K)_\fm <\infty\,.
\]
It remains to note the adjunction isomorphism
\[
\Hom^*_{\cat T_K}(\alpha_*K, M_K)\cong \Hom^*_{\cat U_K}(K,\alpha^*(M_K))
\]
for $\cat U_K=\StMod K[t]/(t^p)$ and Lemma~\ref{le:cyclic}.
\end{proof}

The result below, implicit in \cite[Section~9]{Benson/Iyengar/Krause/Pevtsova:2018a}, was used in the argument above.

\begin{lemma}  
\label{le:compact-generator}
In the notation of Theorem~\ref{th:pi-main}, the thick subcategory generated by $\alpha_*K$ in  $\StMod KG$ is $\comp{(\gam_\fm\StMod KG)}$. 
\end{lemma}

\begin{proof}
Set $\cat T_K=\StMod KG$ and $\cat U_K = \StMod K[t]/t^p$. The $KG$-module $\alpha_*K$ is finite dimensional and hence compact in $\cat T_K$. It is easy to verify that $\supp(\alpha_*K)=\{\fm\}$; see, for example, the first part of the proof of \cite[Lemma~9.1]{Benson/Iyengar/Krause/Pevtsova:2018a}. Thus $\alpha_*K$ is in $\gam_\fm\cat T_K$. For any $N\ne 0$ in $\gam_\fm\cat T$ one has $\alpha^*N\ne 0$ in $\StMod K[t]/(t^p)$, and hence
\[
\Hom_{\cat T_K}(\alpha_*K, N)\cong \Hom_{\cat U_K}(K, \alpha^*N)\ne 0\,.
\]
Given these facts, a standard argument yields that the localising subcategory generated by $\alpha_*K$ equals $\gam_\fm\cat T_K$. Since $\alpha_*K$ is compact, the desired result follows.
\end{proof}

\begin{ack}
Our work on this project started many years ago and our perspective on it has evolved considerably over this time. We have benefited from conversations with numerous colleagues.  It is a pleasure to thank in particular Greg Stevenson for sharing his insights. 

Part of this work was done during the Trimester Program `Spectral Methods in Algebra, Geometry, and Topology' in 2022 at the Hausdorff Institute in Bonn. It is a pleasure to thank HIM for hospitality and for funding of all four authors by the Deutsche Forschungsgemeinschaft under Excellence Strategy EXC-2047/1-390685813.  We thank the American Institute for Mathematics for its hospitality and the support through the SQuaRE program. Some of the material is based on work supported by the National Science Foundation under Grant No. DMS-1928930 and by the Alfred P. Sloan Foundation under grant G-2021-16778, while three of the authors were in residence at the Simons Laufer Mathematical Sciences Institute (formerly MSRI) in Berkeley, California, during the Spring 2024 semester. Another three of us would like to thank the Oberwolfach Research Institute for Mathematics for their hospitality in September of 2023.

SBI was partly supported by NSF grant DMS-2001368, HK was partly supported by the Deutsche Forschungsgemeinschaft (SFB-TRR 358/1 2023 - 491392403), and JP was partly supported by NSF grants DMS-1901854, 2200832 and the  Brian and Tiffinie Pang faculty fellowship.
\end{ack}

\bibliographystyle{amsplain}
\bibliography{repcoh}

\end{document}